\documentclass[11pt,reqno]{amsart}
\usepackage{}
\usepackage{mathrsfs}
\usepackage{amsfonts}
\usepackage{a4wide}
\allowdisplaybreaks \numberwithin{equation}{section}
\usepackage{color}

\newtheorem{thm}{Theorem}[section]

\newtheorem{lem}[thm]{Lemma}
\newtheorem{rem}[thm]{Remark}

\newtheorem{prop}[thm]{Proposition}

\theoremstyle{definition}

\newcommand{\R}{\mathbb{R}}
\newcommand{\ds}{\displaystyle}

\begin{document}
\title[A new type of solutions]
{infinitely many new solutions for a nonlinear coupled Schr\"odinger system }
 \author{ Qingfang Wang and Mingxue Zhai$^\dagger$}
\address[Qingfang Wang]{School of Mathematics and Computer Science, Wuhan Polytechnic University, Wuhan, 430023, P. R. China }
\email{wangqingfang@whpu.edu.cn}

\address[Mingxue Zhai]{School of Mathematics and Statistics, Central China Normal University, Wuhan, 430079, P. R. China}
\email{mxzhai@mails.ccnu.edu.cn}
\thanks{}
\begin{abstract}
We revisit the following nonlinear Schr\"odinger system
\begin{equation*}
\begin{cases}
-\epsilon^{2}\Delta u +P(x) u= \mu_1 u^3  +\beta uv^2,
&~\text{in}\;\mathbb R^3,\\
\noalign{\vskip1truemm} -\epsilon^{2}\Delta v+Q(x) v= \mu_2 v^3 +\beta
u^2v, &~\text{in}\;\mathbb R^3,
\end{cases}
\end{equation*}
where $\epsilon$ is a positive parameter, $P(x),\,Q(x)$ are the potential functions, $\mu_1>0$, $\mu_2>0$ and $\beta\in\mathbb R$ is
 a coupling constant.
Employing  the finite dimensional reduction method, we prove that there are new kind of
synchronized and segregated solutions, which concentrate both in a bounded domain and near infinity, and present a special structure.
Moreover, by applying the local Pohozaev identities and some point-wise estimates of the errors, we prove that the new kind of synchronized solutions are non-degenerate, which is of great interest independently. One of the main difficulties of Schr\"odinger system come from the interspecies interaction between the components, which never appear in the study of single equation. Secondly, prior to the construction of new solutions, we shall verify the non-degeneracy of the solutions established in [Peng-Pi, Discrete Contin. Dyn. Syst., 2016] for the Schr\"odinger systems.
\end{abstract}
\thanks{$\dagger$ Corresponding author}

\maketitle
 \small{
\section{Introduction}

We consider the following nonlinear Schr\"odinger system
\begin{equation}\label{eq0}
\begin{cases}
-\epsilon^{2}\Delta u +P(x) u= \mu_1 u^3  +\beta uv^2,
&~\text{in}\;\mathbb R^3,\\
\noalign{\vskip1truemm} -\epsilon^{2}\Delta v+Q(x) v= \mu_2 v^3 +\beta
u^2v, &~\text{in}\;\mathbb R^3,
\end{cases}
\end{equation}
where  $P(x),\, Q(x)$ are continuous positive
functions, $\mu_1>0$, $\mu_2>0$ and $\beta\in\mathbb R$ is
 a coupling constant.

These type of systems arise when one considers standing wave solutions of time dependent $N$-coupled Schrodinger systems with $N=2$ of the form
\begin{equation}\label{eq00}
\begin{cases}
 -i\partial_t\Phi_j=\frac{h^2}{2l}\Delta\Phi_j-V_j(x)\Phi_j-\mu_j|\Phi_j|^2\Phi_j-\sum_{i\neq j}\beta_{ij}|\Phi_i|^2\Phi_j,\ &x\in\R^3,t>0,\\
 \Phi_j=\Phi(x,t)\in\mathbb C, \ \ &j=1,2
\end{cases}
\end{equation}
these systems, also known as the Gross-Pitaevskii system \cite{EG,EGBB,hmewc,pw}, serves as a mathematical model for binary Bose-Einstein condensate, describing the time evolution of the condensate wave functions $\Phi_j$, where $j=1,2$. Here, $h$ is the Planck constant, the constants $\mu_j$ and $\beta_{ij}$ represent the intraspecies and interspecies scattering lengths respectively, which characterize interactions between identical and distinct particle species, the functions $V_j(x)$  denote the magnetic trapping potentials. It is natural to assume that $\beta_{ij}$ is symmetric(i.e., $\beta_{ij}=\beta_{ji}$ if $j\neq i$).
The sign of the scattering length $\beta=\beta_{ij}$ determines the interaction of two states.
When $\beta>0$ indicates an attractive interaction, where the components of the vector solution come together. While $\beta<0$ indicates a repulsive interaction, where the components turn into repel each other and form phase separations.

\smallskip
The existence and qualitative properties of solutions for nonlinear Schr\"odinger equations have been the subject of extensive investigation from a mathematical perspective in recent years. Various aspects, such as the existence and asymptotic behavior of ground states or bound states, have been explored in numerous studies \cite{3,8,11,13,15,21,LW05,23,26,31,32,ZZR} with semiclassical states established in \cite{12,20,24,29,WW}.
In the absence of trapping potentials, specifically with constant potentials as examined in \cite{19}, it has been demonstrated that when the coupling constant $\beta$ is negative and sufficiently large, the two states repel each other, behaving as two separate spikes (see \cite{14,34}).
Conversely, when $\beta$ is positive and sufficiently large, spikes of states may attract each other, behaving like a single spike.

With the influence of trapping potentials, as analyzed in \cite{20}, spikes are separated and trapped at the minimum points of potentials when $\beta<0$. On the contrary, for $\beta>0$, there is a competition between the attractions of spikes and trap potential wells. If the potential wells are strongly enough, the spikes separate. However, if the attraction of potential wells is not strong enough, the spikes come together. These observations indicate that the interspecies scattering length and potentials can significantly impact the locations of spikes, making the interaction of spikes complex and challenging to determine.
For additional references on this aspect, one can refer to  \cite{ambrosetti-colorado-ruiz,BC13,GLWZ19-1,GLWZ19-2,GLW,GS,Guo1,
LWW-2024-JDE,LWW-2024-JDE2,LP14,PV,
PWW-2019-TAMS,RL14}.
\smallskip

When $\epsilon=1$ in system \eqref{eq0}, Peng and Wang \cite{PW13} constructed an unbounded sequence of non-radial positive vector
solutions $(u_{l}(x),v_{l}(x))$ in the
repulsive case $(\beta>0)$, which exhibit a segregated type, and an unbounded
sequence of non-radial positive vector solutions of synchronized type in the attractive case $(\beta<0)$.

Peng and Pi \cite{PP16} investigated how potentials and the interspecies scattering length $\beta$ influence the solutions structure of system \eqref{eq0}, extending the results of \cite{20} for least energy solutions.  They established the existence of higher energy solutions to system \eqref{eq0}, not only  providing the locations of spikes but also intricately depicting the interaction of spikes.

This prompts a natural question: Can we glue these two sets of solutions from \cite{PP16} and \cite{PW13} to yield a new type of solution? Specifically, we aim to construct a new type of solutions to system \eqref{eq0} with a shape given, at the main order, by
 \begin{equation}\label{eqs2}
 \Big(u_{l}(x)+\sum_{j=1}^{k}U(\frac{x-\eta^{j}}{\epsilon}),
 v_{l}(x)+\sum_{j=1}^{k}V(\frac{x-\eta^{j}}{\epsilon})\Big),
 \end{equation}
 which
 concentrates both in a bounded domain and near infinity, where $\Big(\sum_{j=1}^{k}U(\frac{x-\eta^{j}}{\epsilon}),
\sum_{j=1}^{k}V(\frac{x-\eta^{j}}{\epsilon})\Big)$ is defined in \textbf{Theorem A} later.

Our construction requires that the two solution families from \cite{PP16}  and \cite{PW13} can be glued together, requiring that the potential functions $P$ and $Q$ are radially symmetric about some fixed point. Therefore, we focus on  Theorem 1.1 in \cite{PP16} (i.e., Theorem A). In this case, the two components of the solution concentrate at the same critical points of the two potential functions. Without loss of generality, we assume this point is the origin, allowing us to combine these solutions with those concentrated at $k$ symmetric points on a coordinate plane established in \cite{PW13}. However, the solutions constructed in Theorem 1.2 of \cite{PP16}, which concentrate at different critical points of $P$ and $Q$, do not satisfy the conditions for gluing and cannot yield new solutions through gluing.
\smallskip
Before delving into the details, let us revisit the following well-known results. Denote the unique radial solution of the following problem by $W$,
\begin{equation*}
\begin{cases}
-\Delta w +   w = w^3, & \text{in}\; \mathbb R^3,\\
w(0)=\displaystyle\max_{x\in\mathbb R^3}w(x),& w\in H^1(\mathbb
R^3),
\end{cases}
\end{equation*}
and denote $W_{\mu_i}$ as the solution of the following problem
\begin{equation}\label{eqs1.2}
\begin{cases}
-\Delta w +   w = \mu_iw^3, & \text{in}\; \mathbb R^3,\\
w(0)=\displaystyle\max_{x\in\mathbb R^3}w(x),& w\in H^1(\mathbb
R^3).
\end{cases}
\end{equation}
It is well known that $W_{\mu_i}$ is non-degenerate and $W_{\mu_i}(x)=W_{\mu_i}(|x|)$.
Moreover, one finds that $(U,V)=(\gamma_1 W,\gamma_2 W)$ solves
\begin{equation}\label{eqQ1}
\begin{cases}
-\Delta U + U= \mu_1 U^3  +\beta UV^2,  &~\text{in}\;\mathbb R^3,\\[1mm]
-\Delta V+ V= \mu_2 V^3  +\beta U^2V, &~\text{in}\;\mathbb R^3
\end{cases}
\end{equation}
 provided $\beta\in \big(-\sqrt{\mu_1\mu_2},\min\{\mu_1,\mu_2\}\big)
 \cup \big(\max\{\mu_1,\mu_2\},+\infty\big)$
, where
$$\gamma_1=\sqrt{\frac{\mu_2-\beta}{\mu_1\mu_2-\beta^2}},\ \ \gamma_2=\sqrt{\frac{\mu_1-\beta}{\mu_1\mu_2-\beta^2}}.$$
\medskip

 To introduce our main motivation, we first state a existence result of multi-peak solutions for system \eqref{eq0} with $\epsilon\rightarrow 0$
under the following assumptions $P(x),Q(x)$ and $\beta$:
\smallskip

${\bf(A_1)}$ $P(x),Q(x)\in C^{1}(\R^{3},\R),$
there exists a positive constant $\alpha$
such that $P(x),Q(x)\geq \alpha$ for all $x\in \R^{3}.$

${\bf(A_2)}$ there exists $ \delta>0$ and $y_{0}\in \R^{3}$  such that $P(x)<P(y_{0}),Q(x)<Q(y_{0})$
for any $x\in B_{\delta}(y_{0})\setminus\{y_{0}\}$ and $P(y_{0})=Q(y_{0})=1,$
where $B_{\delta}(y_{0})=\{x\in \R^{3}: |x-y_{0}|<\delta\}.$

${\bf(A_\beta)}$ There exists
 an integer $\beta^{*}>0$ and a decreasing sequence $\{\beta_{n}^{*}\}\subset (-\sqrt{\mu_{1}\mu_{2}},0)$
 with $\lim_{n\rightarrow\infty}\beta_{n}^{*}=-\sqrt{\mu_{1}\mu_{2}},$
 such that
 $\beta\in ((-\beta^{*},0)\setminus  \{\beta_{n}^{*}\}_{n=1}^{\infty}\cup (0,\min\{\mu_{1},\mu_{2}\})\cup (\max\{\mu_{1},\mu_{2}\},\infty)$.
\smallskip

\noindent\textbf{Theorem A.} (c.f \cite{PP16})
 \emph{
 Suppose that $P(x),Q(x)$ satisfy $(A_1)$ and $(A_{2})$ and $\beta$ satisfies $(A_\beta)$. Then
 for any positive integer $n\geq 2$, system \eqref{eq0}
 has a positive solution $(u_{\epsilon},v_{\epsilon})$ with $n$ spikes concentrating near $y_{0}$ for $\epsilon$ sufficiently small.}
 Specifically, $(u_{\epsilon},v_{\epsilon})$ has the following form
 $$
 (u_{\epsilon},v_{\epsilon})=\Big(\sum_{j=1}^{n}U(\frac{x-\eta^{j}}{\epsilon})+\omega_{\epsilon}(x),
 \sum_{j=1}^{n}V(\frac{x-\eta^{j}}{\epsilon})+\varpi_{\epsilon}(x))\Big),
 $$
 where $(\omega_{\epsilon},\varpi_{\epsilon})\in H$ with
 $$
 \|(\omega_{\epsilon},\varpi_{\epsilon})\|=O\big(\epsilon^{5}|ln \epsilon|^{2}\big),\,\,\frac{|\eta^{i}-\eta^{j}|}{\epsilon}\geq |ln \epsilon|^{\frac{1}{2}}(i\neq j),\,\,\eta^{j}\in B_{\delta}(y_{0})(j=1,...,n).
 $$
\smallskip

First we are aimed to prove the non-degeneracy property of the solution $(u_{\epsilon},u_{\epsilon})$
   to \eqref{eq0} constructed in  Theorem A,
 in the sense that the linearized operator $L_\epsilon$ defined by
 \begin{equation*}\label{Lk}
\begin{split}
\big\langle L_\epsilon (\psi_1,\psi_2),(\phi_1,\phi_2)\big\rangle
  :=&\int_{\mathbb R^3}
\big[\epsilon^2\nabla  \psi_1\nabla\phi_1 + P(x)\psi_1\phi_1
+\epsilon^2\nabla  \psi_2\nabla\phi_2 + Q(x)\psi_2\phi_2
\big]\\
&-3 \int_{\mathbb R^3}\big[\mu_1 u_\epsilon^2\psi_1\phi_1+\mu_2 v_\epsilon^2\psi_2\phi_2\big]-
\int_{\mathbb R^3}
\big[2\beta
u_\epsilon v_\epsilon \psi_2\phi_1+\beta
v_\epsilon^2\psi_1\phi_1 \big]\\
&-\int_{\mathbb R^3}
\big[2\beta
u_\epsilon v_\epsilon\psi_1\phi_2+\beta
u_\epsilon^2\psi_2\phi_2\big]
\end{split}
\end{equation*}
has trivial kernel in $ H^1(\R^3)\times H^1(\R^3)$. We note that the
non-degenerate result play an essential role to construct solutions
when Lyapunov Schmidt reduction are used to  the systems. For the
non-degeneracy and the local uniqueness results of
Schr\"{o}dinger equations or systems, one can refer to
\cite{guo-musso-peng-yan2,GPY,LPY19,NY00} etc.

For this purpose, we need to impose the conditions for the potential functions $P(x)$ and $Q(x)$ as follows

$\bf (H_{1})$:
$0$ is a critical of the radial functions $P(x),\,Q(x)$ and $ ~\gamma_{1}^{2}P(x)+\gamma_{2}^{2}Q(x)$
is non-degenerate at $0$.

${\bf (H_2)}$: For $i=1,2$, there are constants $m_i>2$, $a_i\neq0, b_i>0$ and $\theta>0$ such that as $r\rightarrow+\infty$
\begin{align*}
\begin{split}
&P(r)=1+\frac{a_1}{r^{m_1}}+\frac{b_1}{r^{m_1+1}}+O(\frac 1{r^{m_1+1+\theta}}),\ \
Q(r)=1+\frac{a_2}{r^{m_2}}+\frac{b_2}{r^{m_2+1}}+O(\frac
1{r^{m_2+1+\theta}}),
\end{split}
\end{align*}
where we further suppose
\begin{align*}
\begin{cases}
a_1>0,&\text{if} ~~m_1<m_2,\\
a_2>0,&\text{if}~~ m_1>m_2,\\
a_1\gamma_1^2+a_2\gamma_2^2>0,&\text{if}~~ m_1=m_2.
\end{cases}
\end{align*}
In the following we always denote $m:=\min\{m_1,m_2\}$.

Set
\begin{align*}\label{xij}
x^{j}=\Big(r\cos\frac{2(j-1)\pi}{k},r\sin\frac{2(j-1)\pi}{k},0\Big), \ j=1,\ldots,k,
\end{align*}
where $r\in [r_0 k\ln k,r_1 k\ln k]$ for some $r_1>r_0>0$.
Let
\begin{align}
U_{\epsilon, x^j}:=U_\epsilon(x-x^j)=U(\frac{x-x^j}{\epsilon}),\,\,\,\,\,V_{\epsilon, x^j}:=V_\epsilon(x-x^j)=V(\frac{x-x^j}{\epsilon}).
\end{align}
Also, we define
\begin{align*}
H_{\epsilon,P}=\Big\{u\in H^1(\R^3):&\,\,\,\,\,\,\,u \hbox{\ is\ even\ in}\ y_2,y_3,\,\cr
 &u(r\cos\theta,r\sin\theta,y'')=u\Big(r\cos(\theta+\frac{2\pi j}k),r\sin(\theta+\frac{2\pi j}k),y''\Big)
\Big\}.
\end{align*}
where the norm of $H^1(\R^N)$ is induced by the inner product
\begin{align*}
\langle u,v \rangle :=\Big(\int_{\R^3}\epsilon^2\nabla u\nabla v+P(x)uv\Big),u,v\in H^1(\R^N).
\end{align*}

Moreover, denote $$H:=H_{\epsilon,P}(\R^3)\times H_{\epsilon,Q}(\R^3)$$ with the norm
\begin{align*}
\|(u,v)\|^{2}:=\|u\|^{2}_{\epsilon,P}+\|u\|^{2}_{\epsilon,Q}.
\end{align*}
\vskip 0.2cm

The first main result of this paper is as follows.
\begin{thm}\label{th3}
Assume that \textup{($A_1)-(A_2),\,(H_1)-(H_2$)}  and $(A_\beta)$ hold.
 Let $(u_{\epsilon},v_{\epsilon})$ be a pair of solution obtained in Theorem A. Then
 system \eqref{eq0} has infinitely many pairs of synchronized dichotomous concentrated solutions.
 More specifically, for any integer $k>k_{0}$ large enough, system \eqref{eq0}  has a pair
 of solution of the following form
 $$
 u_{\epsilon,k}=S_{\epsilon,k}(x)+\varphi_{k}(x),v_{\epsilon,k}=T_{\epsilon,k}(x)+\psi_{k}(x)
 $$
where
 $$
 S_{\epsilon,k}(x):=u_{\epsilon}(x)+\sum_{j=1}^{k}U_{\epsilon,x^{j}}(x),\,T_{\epsilon,k}(x):=v_{\epsilon}(x)+\sum_{j=1}^{k}V_{\epsilon,x^{j}}(x),
 $$
and $\|(\varphi_{k},\psi_{k})\| \rightarrow 0$ as $k\rightarrow\infty$.
\end{thm}
\begin{rem}
In the proof of Theorem \ref{th3}, the process of constructing solutions that are concentrated respectively at the origin and at infinity is not achieved by directly superimposing spikes centered at the origin and in the vicinity of infinity as an approximate solution for reduction. Instead, it is accomplished by first proving that the solutions concentrated at the origin  as $\epsilon\rightarrow0$  described in Theorem A are non-degenerate. Then, by fixing a sufficiently small $\epsilon>0$, these concentrated solutions are combined with k spikes in the vicinity of infinity to form a new approximate solution. This process involves taking the limit as $k$ approaches infinity and reconstructing solutions concentrated at infinity.
\end{rem}

Define
\begin{align*}
y^j=\Big(\rho\cos\frac{(2j-1)\pi}{k},\rho\sin\frac{(2j-1)\pi}{k},0\Big), \ j=1,\ldots,k,
\end{align*}
where $\rho\in[c_{1}k\ln k,c_{2}k\ln k]$ for $0<c_{1}<c_{2}$. Also we write
$$
W_{\epsilon,\mu_1,x^j}:=W_{\mu_1}(\frac{|x-x^j|}{\epsilon}),~~~~~~W_{\epsilon,\mu_2,y^j}:=W_{\mu_2}(\frac{|x-y^j|}{\epsilon}).
$$
We have  another result as follows.
\begin{thm}\label{th1}
 Under the same conditions as in Theorem \ref{th3}
 and $m_1=m_2=m$.
Let $(u_{\epsilon},v_{\epsilon})$ be a pair of solution obtained in Theorem A. Then there exists $\varepsilon>0$ small enough such that $\beta<\varepsilon$,
 system \eqref{eq0} has infinitely many many pairs of segregated dichotomous concentrated solutions.
 More specifically, for any integer $k>k_{0}$ large enough, system \eqref{eq0}  has a pair
 of solution of the following form
 $$
 u_{\epsilon,k}=\bar{S}_{\epsilon,k}(x)+\bar{\varphi}_{k}(x),\,\,v_{\epsilon,k}=\bar{T}_{\epsilon,k}(x)+\bar{\psi}_{k}(x),
 $$
where
 $$
 \bar{S}_{\epsilon,k}(x)=u_{\epsilon}(x)+\sum_{j=1}^{k}W_{\epsilon,\mu_1,x^{j}}(x),\,\,\bar{T}_{\epsilon,k}(x):=v_{\epsilon}(x)+\sum_{j=1}^{k}W_{\epsilon,\mu_2,y^{j}}(x),
 $$
 where $\|(\bar{\varphi}_{k},\bar{\psi}_{k})\| \rightarrow 0$ as $k\rightarrow\infty.$
\end{thm}

Inspired by \cite{GZ-2022},  we are aimed to discuss the non-degeneracy of $(u_{\epsilon,k},v_{\epsilon,k})$
obtained in Theorem \ref{th3}, which is of great importance  to construct more new type of  concentrated solutions by the same gluing procedure.
To this end, we also define the   linearized operator $L_k$ as
 \begin{equation*}\label{Lk}
\begin{split}
\big\langle  L_k (\psi_1,\psi_2),(\phi_1,\phi_2)\big\rangle
  :=&\int_{\mathbb R^3}
\big[\epsilon^2\nabla  \psi_1\nabla\phi_1 + P(x)\psi_1\phi_1
+\epsilon^2\nabla  \psi_2\nabla\phi_2 + Q(x)\psi_2\phi_2\big]
\\
&-3 \int_{\mathbb R^3}\big[\mu_1 u_{\epsilon,k}^2\psi_1\phi_1+\mu_2 v_{\epsilon,k}^2\psi_2\phi_2\big]
-\int_{\mathbb R^3}
\big[2\beta
u_{\epsilon,k} v_{\epsilon,k} \psi_2\phi_1+\beta
v_{\epsilon,k}^2\psi_1\phi_1 \big]\\
&-\int_{\mathbb R^3}\big[2\beta
u_{\epsilon,k} v_{\epsilon,k}\psi_1\phi_2+\beta
u_{\epsilon,k}^2\psi_2\phi_2\big].
\end{split}
\end{equation*}

${\bf ( H'_2)}$: Under $(H_2)$, we further assume that
\begin{align*}
\begin{split}
&
P'(r)=-\frac{a_1m_1}{r^{m_1+1}}-\frac{b_1(m_1+1)}{r^{m_1+2}}+O(\frac 1{r^{m_1+2+\theta}}),\ \
Q'(r)=-\frac{a_2m_2}{r^{m_2+1}}-\frac{b_2(m_2+1)}{r^{m_2+2}}+O(\frac
1{r^{m_2+2+\theta}}).
\end{split}
\end{align*}

The non-degeneracy of $(u_{\epsilon,k},v_{\epsilon,k})$ can be stated as the following.
\begin{thm}\label{th1.3}
 Assume that \textup{($A_{1}),(A_2),(H_1),( H'_2)$}  and $(A_\beta)$ hold.  Let
 $(\xi_{1},\xi_{2})\in(H_{\epsilon,P}\cap H^1(\mathbb{R}^3))\times (H_{\epsilon,Q}\cap H^1(\mathbb{R}^3))$  satisfies
 $L_k(\xi_{1},\xi_{2})=0$. Then $(\xi_{1},\xi_{2})=0$ for fixed small $\epsilon$.
\end{thm}

\begin{rem}

In comparison to exploring the existence of concentrated solutions, the challenges introduced by the coupling terms in the system become notably more intricate when investigating the non-degeneracy result.  Delicate estimates becomes imperative for the interspecies interaction among the components of system \eqref{eq0}, a consideration absent in the study of single equations.
Our approach incorporates innovative observations, drawing inspiration from the characterization of Fermat points in the renowned Fermat problem. For a more comprehensive understanding, refer to the discussions in Section 3.

Additionally, unlike the approach in \cite{GZ-2022}, the proof of Theorem \ref{th1.3} necessitates the selection of distinct weighted norms, owing to differences in the form of the concentrated solutions.
\end{rem}

\begin{rem}
Recognizing the increased difficulty in establishing the non-degeneracy of the solution obtained in \ref{th1}, we plan to address this aspect in an upcoming paper.
\end{rem}
\medskip

The structure of this paper is as follows: firstly, in Section 2, we prove the non-degeneracy of the concentrated solutions obtained in Theorem A. Subsequently, utilizing this non-degeneracy result, in Section 3, we provide the proofs for Theorem \ref{th3} and Theorem \ref{th1} by constructing two types of dichotomy, namely, synchronized and segregated dichotomous solutions.
 Section 4 is dedicated to further discussing the non-degeneracy of the synchronized dichotomous solutions constructed in the Theorem \ref{th3}. Technical results and energy estimates are put in the appendix.

\medskip
\section{Non-degeneracy of the solutions in Theorem A}

In this part, we mainly consider the non-degenerate result to
the solutions obtained in  Theorem A, which is very crucial for us to construct these solutions in Theorems \ref{th3} and \ref{th1}.

Note that  $(u_{\epsilon},v_{\epsilon})$ solves system \eqref{eq0}, and $(\xi_{1\epsilon},\xi_{2\epsilon})$ satisfies
\begin{equation*}\label{eq1}
\begin{cases}
-\epsilon^{2}\Delta \xi_{1\epsilon} +P(|y|) \xi_{1\epsilon}=3\mu_1 u_\epsilon^2\xi_{1\epsilon}  +\beta \xi_{1\epsilon}v_\epsilon^2 +2\beta u_\epsilon v_\epsilon\xi_{2\epsilon}, &~\text{in}\;\mathbb R^3,\\[3mm]
-\epsilon^{2}\Delta \xi_{2\epsilon}+Q(|y|) \xi_{2\epsilon}= 3\mu_2 v_\epsilon^2\xi_{2\epsilon}  +\beta u_\epsilon^2\xi_{2\epsilon} +2\beta u_\epsilon v_\epsilon\xi_{1\epsilon}, &~\text{in}\;\mathbb R^3.
\end{cases}
\end{equation*}
Then applying the standard method (see \cite{GZ-2022} for example), we have the following Pohozaev identity.
\begin{lem}\label{lemlpoh1}
For $i=1,2,3$ and any $\Omega\subset\R^3$, it holds that
\begin{align*}
-&\int_{\partial\Omega}\epsilon^{2}\Big(\frac{\partial u_{\epsilon}}{\partial\nu}\frac{\partial \xi_{1\epsilon}}{\partial y_i}+\frac{\partial \xi_{1\epsilon}}{\partial\nu}\frac{\partial u_{\epsilon}}{\partial y_i}
+\frac{\partial v_{\epsilon}}{\partial\nu}\frac{\partial \xi_{2\epsilon}}{\partial y_i}+\frac{\partial \xi_{2\epsilon}}{\partial\nu}\frac{\partial v_{\epsilon}}{\partial y_i}\Big)
+\epsilon^{2}\int_{\partial\Omega}\big(\langle\nabla u_{\epsilon},\nabla\xi_{1\epsilon}\rangle\nu_i+\langle\nabla v_{\epsilon},\nabla\xi_{2\epsilon}\rangle\nu_i\big)\\
 +&\int_{\partial\Omega}(P(|y|)u_{\epsilon}\xi_{1\epsilon}+Q(|y|)v_{\epsilon}\xi_{2\epsilon})\nu_i-\int_{\partial\Omega}
\big(\mu_1 u_{\epsilon}^3\xi_{1\epsilon} +\mu_2 v_{\epsilon}^3\xi_{2\epsilon} +\beta u_{\epsilon}\xi_{1\epsilon}v_{\epsilon}^2 +\beta \xi_{2\epsilon}v_{\epsilon}u_{\epsilon}^2\big)\nu_i\\
&=\int_\Omega\Big(\frac{\partial{P}}{\partial y_i}u_{\epsilon}\xi_{1\epsilon}+\frac{\partial{Q}}{\partial y_i}v_{\epsilon}\xi_{2\epsilon}\Big).
\end{align*}

\end{lem}

The main result in this section is as follows:
\begin{prop}\label{lema.1}
 Under the same assumptions imposed on $P(x),\,Q(x)$ as Theorem 1.1,
 and suppose that $(u_{\epsilon},v_{\epsilon})$ is a positive solution of
 system\eqref{eq0} obtained in Theorem A.  Let $(\xi_{1\epsilon},\xi_{2\epsilon})$  be a solution of
 $L_\epsilon(\xi_{1\epsilon},\xi_{2\epsilon})=0$. Then $(\xi_{1\epsilon},\xi_{2\epsilon})=0$.
\end{prop}

\begin{proof}
We argue by contradiction, we suppose that there exist
$\epsilon_n\rightarrow 0$ as $n\rightarrow \infty$, satisfying $(\xi_{1\epsilon_n},\xi_{2\epsilon_n})\in
( H^1(\R^3))^2, \|(\xi_{1\epsilon_n},\xi_{2\epsilon_n})\|_{L^{\infty}(\R^{3})}=1$ and
$L_{\epsilon_n}(\xi_{1\epsilon_n},\xi_{2\epsilon_n})=0$.

Fix $j\in \{1,\cdots,m\}$ and let for $i=1,2, $ $$\xi_{i \epsilon_{n}}^{j}(y)=\xi_{i \epsilon_{n}}(\epsilon_{n} y+\eta^{j}).$$

Step 1. We claim $$
(\xi_{1 \epsilon_{n}}^{j}(y),\xi_{2 \epsilon_{n}}^{j}(y))\rightarrow \sum_{l=1}^3b_j^l(\frac{\partial U}{\partial y_l},
\frac{\partial V}{\partial y_l})
$$
uniformly in $C^1(B_R(0))$ for any $R>0$, where  $b_j^l(j=1,2,...,m)$
are some constants.

By assumption we get $|\xi^{j}_{i \epsilon_{n}}|\leq 1$, and we may assume that $\xi^{j}_{i \epsilon_{n}}\rightarrow\xi^{j}_i$ in $C_{loc}(\R^3)$.
 In view of $L_{\epsilon_{n}}(\xi_{1\epsilon_{n}},\xi_{2\epsilon_{n}})=0$ and
 $$\xi^{j}_{i \epsilon_{n}}(y)=\xi_{i \epsilon_{n}}(\epsilon_{n} y+\eta^{j}),(i=1,2;j=1,2,...,m),$$
 we know that
 \begin{eqnarray}\label{eq1}
\begin{cases}
-\Delta \xi^{j}_{1\epsilon_{n}}(y) +P(\epsilon_{n} y+\eta^{j}) \xi^{j}_{1\epsilon_{n}}(y)=3
\mu_1 u^2_{\epsilon_{n}}(\epsilon_{n} y+\eta^{j})\xi^{j}_{1\epsilon_{n}}(y)
+\beta \xi^{j}_{1\epsilon_{n}}(y)v^2_{\epsilon_{n}}(\epsilon_{n} y+\eta^{j})\\
 \quad \quad \quad \quad \quad \quad \quad \quad \quad \quad \quad \quad \quad \quad \quad \quad \quad+2\beta u_{\epsilon_{n}}(\epsilon_{n} y+\eta^{j})v_{\epsilon_{n}}(\epsilon_{n} y+\eta^{j})\xi^{j}_{2\epsilon_{n}}(y),\\
-\Delta \xi^{j}_{2\epsilon_{n}}(y)+Q(\epsilon_{n} y+\eta^{j}) \xi^{j}_{2\epsilon_{n}}(y)= 3\mu_2 v_{\epsilon_{n}}^2(\epsilon_{n} y+\eta^{j})\xi^{j}_{2\epsilon_{n}}(y)  +
\beta u_{\epsilon_{n}}^2(\epsilon_{n} y+\eta^{j})\xi^{j}_{2\epsilon_{n}}(y)\\
\quad \quad \quad \quad \quad \quad \quad \quad \quad \quad \quad \quad \quad \quad \quad \quad \quad +2\beta u_{\epsilon_{n}}(\epsilon_{n} y+\eta^{j}) v_{\epsilon_{n}}(\epsilon_{n} y+\eta^{j})\xi^{j}_{1\epsilon_{n}}(y).
\end{cases}
\end{eqnarray}
 Then as $\epsilon_{n}\rightarrow 0,$ in \eqref{eq1}, $(\xi^j_1,\xi^j_2)$
satisfies
\begin{align}\label{eqxi}
\begin{cases}
\displaystyle-\Delta\xi_1^{j}+\xi_1^{j}=3\mu_1 U^2\xi_1^{j}+\beta V^2\xi_1^{j}
+2\beta UV\xi_2^{j},&~\text{in}\;\mathbb R^3,\\[1mm]
\displaystyle-\Delta\xi_2^{j}+\xi_2^{j}=3\mu_2 V^2\xi_2^{j}
+\beta U^2\xi_2^{j}+2\beta UV\xi_1^{j},&~\text{in}\;\mathbb R^3,
\end{cases}
\end{align}
which gives
$$(\xi^j_1,\xi^j_2)= \sum_{l=1}^3b_j^l\big(\frac{\partial U}{\partial y_l},
\frac{\partial V}{\partial y_l}\big).$$
Therefore we complete the proof of the claim.

Step 2. It holds that
$\xi^{j}_{i \epsilon_{n}}\rightarrow0,\ i=1,2$
uniformly in $C^1(B_R(0))$ for any $R>0$.

Let $\delta>0$ be a fixed small constant. We apply  Lemma \ref{lemlpoh1} with
$\Omega=B_{\delta}(\eta^{j})$.
Define the following bilinear form
\begin{align}
&\mathcal L(\vec{\textbf{u}}_{\epsilon_{n}},\vec\xi,B_{\delta}(\eta^{j}))\cr
&=
-\epsilon_{n}^{2}\int_{\partial B_{\delta}(\eta^{j})}\Big(\frac{\partial
u_{\epsilon_{n}}}{\partial\nu}\frac{\partial \xi_{1\epsilon_{n}}}{\partial y_l}+\frac{\partial
\xi_{1\epsilon_{n}}}{\partial\nu}\frac{\partial u_{\epsilon_{n}}}{\partial y_l}
+\frac{\partial v_{\epsilon_{n}}}{\partial\nu}\frac{\partial \xi_{2\epsilon_{n}}}{\partial y_l}+\frac{\partial \xi_{2\epsilon_{n}}}{\partial\nu}\frac{\partial v_{\epsilon_{n}}}{\partial y_l}\Big)\cr
&+\epsilon_{n}^{2}\int_{\partial B_{\delta}(\eta^{j})}\langle\nabla
u_{\epsilon_{n}},\nabla\xi_{1\epsilon_{n}}\rangle\nu_l+\langle\nabla
v_{\epsilon_{n}},\nabla\xi_{2\epsilon_{n}}\rangle\nu_l
 +\int_{\partial B_{\delta}(\eta^{j})}(u_{\epsilon_{n}}\xi_{1\epsilon_{n}}+v_{\epsilon_{n}}\xi_{2\epsilon_{n}})\nu_l,
 \end{align}
where $\vec {\textbf{u}}_{\epsilon_{n}}=(u_{\epsilon_{n}},v_{\epsilon_{n}}),\vec\xi=(\xi_{1\epsilon_{n}},\xi_{2\epsilon_{n}})$, and $\nu$ is the outward unit normal of $\partial B_{\delta}(\eta^{j})$.
Then Lemma \ref{lemlpoh1} implies that
\begin{align}\label{a.4.1}
&\int_{B_{\delta}(\eta^{j})}\Big(\frac{\partial{P}}{\partial y_l}u_{\epsilon_{n}}\xi_{1\epsilon_{n}}+\frac{\partial{Q}}{\partial y_l}v_{\epsilon_{n}}\xi_{2\epsilon_{n}}\Big)\cr
&=\mathcal L(\vec{\textbf{u}}_{\epsilon_{n}},\vec\xi,B_{\delta}(\eta^{j}))
+\int_{\partial B_{\delta}(\eta^{j})}((P(|y|)-1)u_{\epsilon_{n}}\xi_{1\epsilon_{n}}
+(Q(|y|)-1)v_{\epsilon_{n}}\xi_{2\epsilon_{n}})\nu_l\cr
&-\int_{\partial B_{\delta}(\eta^{j})}
(\mu_1u_{\epsilon_{n}}^3\xi_{1\epsilon_{n}} +\mu_2 v_{\epsilon_{n}}^3\xi_{2\epsilon_{n}} +\beta u_{\epsilon_{n}} \xi_{1\epsilon_{n}}v_{\epsilon_{n}}^2 +\beta \xi_{2\epsilon_{n}}v_{\epsilon_{n}}u_{\epsilon_{n}}^2)\nu_l.
\end{align}
 Also, from Lemmas  \ref{lemb.1} to \ref{lemb.3}, it is easy to prove that there exists a small positive constant $\sigma$ such that
\begin{align}\label{a.4.2}
\mathcal L(\vec{\textbf{u}}_{\epsilon_{n}},\vec\xi,B_{\delta}(\eta^{j}))
=O(e^{-\frac{\sigma}{\epsilon_{n}}}),
\end{align}

\begin{align}\label{a.4.3}
\int_{\partial\Omega}\big[(P(|y|)-1)u_{\epsilon_{n}}\xi_{1\epsilon_{n}}
+(Q(|y|)-1)v_{\epsilon_{n}}\xi_{2\epsilon_{n}}\big]\nu_i=O(e^{-\frac{\sigma}{\epsilon_{n}}})
\end{align}
and
\begin{align}\label{a.4.4}
\int_{\partial\Omega}
(\mu_1 u_{\epsilon_{n}}^3\xi_{1\epsilon_{n}} +\mu_2 v_{\epsilon_{n}}^3\xi_{2\epsilon_{n}} +\beta u_{\epsilon_{n}} \xi_{1\epsilon_{n}}v_{\epsilon_{n}}^2 +\beta \xi_{2\epsilon_{n}}v_{\epsilon_{n}}u_{\epsilon_{n}}^2)\nu_i=O(e^{-\frac{\sigma}{\epsilon_{n}}}).
\end{align}
Since $y_{0}$ is a critical point of $P(x),\,Q(x)$, by Taylor's expansion we have
\begin{align*}
\frac{\partial P(y)}{\partial y_{l}}=\frac{\partial P(y)}{\partial y_{l}}-\frac{\partial P(y_{0})}{\partial y_{l}}
=\sum_{s=1}^{3}(y_{s}-y_{0,s})\frac{\partial^{2} P(y_{0})}{\partial y_{l}\partial y_{s}}
+o(|y-y_{0}|),
\end{align*}
\begin{align*}
\frac{\partial Q(y)}{\partial y_{l}}=\frac{\partial Q(y)}{\partial y_{l}}-\frac{\partial Q(y_{0})}{\partial y_{l}}
=\sum_{s=1}^{3}(y_{s}-y_{0,s})\frac{\partial^{2} Q(y_{0})}{\partial y_{l}\partial y_{s}}
+o(|y-y_{0}|),
\end{align*}
which implies that
\begin{align}\label{a.4.5}
&\int_{B_{\delta}(\eta^{j})}\Big(\frac{\partial{P}}{\partial y_l}u_{\epsilon_{n}}\xi_{1\epsilon_{n}}+\frac{\partial{Q}}{\partial y_l}v_{\epsilon_{n}}\xi_{2\epsilon_{n}}\Big)
\cr
&=\sum_{s=1}^{3}\int_{B_{\delta}(\eta^{j})}(y_{s}-y_{0,s})\frac{\partial^{2} P(y_{0})}{\partial y_{l}\partial y_{s}}u_{\epsilon_{n}}\xi_{1\epsilon_{n}}
+\sum_{s=1}^{3}\int_{B_{\delta}(\eta^{j})}(y_{s}-y_{0,s})\frac{\partial^{2} Q(y_{0})}{\partial y_{l}\partial y_{s}}v_{\epsilon_{n}}\xi_{2\epsilon_{n}}
\cr
&\quad+o\Big(\int_{B_{\delta}(\eta^{j})} |y-y_{0}|( |u_{\epsilon_{n}}\xi_{1\epsilon_{n}}^{j}|+|v_{\epsilon_{n}}\xi_{2\epsilon_{n}}^{j}|)\Big)\cr
&=\epsilon_{n}^{4}\sum_{s=1}^{3}\frac{\partial^{2} P(y_{0})}{\partial y_{l}\partial y_{s}}\int_{B_{\frac{\delta}{\epsilon_{n}}}(0)}y_{s}U\xi_{1\epsilon_{n}}^{j}
+\epsilon_{n}^{4}\sum_{s=1}^{3}\frac{\partial^{2}Q(y_{0})}{\partial y_{l}\partial y_{s}}\int_{B_{\frac{\delta}{\epsilon_{n}}}(0)}y_{s}V\xi_{2\epsilon_{n}}^{j}
+o(\epsilon_{n}^{4})\cr
&=\epsilon_{n}^{4}\sum_{s=1}^{3}\frac{\partial^{2} P(y_{0})}{\partial y_{l}\partial y_{s}}\int_{B_{\frac{\delta}{\epsilon_{n}}}(0)}y_{s}U\sum_{t=1}^3b_j^t\frac{\partial U}{\partial y_t}
+\epsilon_{n}^{4}\sum_{s=1}^{3}\frac{\partial^{2} Q(y_{0})}{\partial y_{l}\partial y_{s}}\int_{B_{\frac{\delta}{\epsilon_{n}}}(0)}y_{s}V\sum_{t=1}^3b_j^t\frac{\partial V}{\partial y_{t}}
+o(\epsilon_{n}^{4})\cr
&=\epsilon_{n}^{4}\sum_{s=1}^{3}\frac{\partial^{2} P(y_{0})}{\partial y_{l}\partial y_{s}}\int_{B_{\frac{\delta}{\epsilon_{n}}}(0)}y_{s}Ub_j^s\frac{\partial U}{\partial y_{s}}
+\epsilon_{n}^{4}\sum_{s=1}^{3}\frac{\partial^{2} Q(y_{0})}{\partial y_{l}\partial y_{s}}\int_{B_{\frac{\delta}{\epsilon_{n}}}(0)}y_{s}Vb_j^{s}\frac{\partial V}{\partial y_{s}}
+o(\epsilon_{n}^{4})\cr
&=\epsilon_{n}^{4}\sum_{s=1}^{3}\gamma^{2}_{1}\frac{\partial^{2} P(y_{0})}{\partial y_{l}\partial y_{s}}\int_{B_{\frac{\delta}{\epsilon_{n}}}(0)}y_{s}Wb_j^s\frac{\partial W}{\partial y_{s}}
+\epsilon_{n}^{4}\sum_{s=1}^{3}\gamma^{2}_{1}\frac{\partial^{2} Q(y_{0})}{\partial y_{l}\partial y_{s}}\int_{B_{\frac{\delta}{\epsilon_{n}}}(0)}y_{s}Wb_j^s\frac{\partial W}{\partial y_{s}}
+o(\epsilon_{n}^{4}).
\end{align}
 From \eqref{a.4.1} to \eqref{a.4.5}, we get
$$
\sum_{s=1}^{3}\Big(\gamma^{2}_{1}\frac{\partial^{2} P(y_{0})}{\partial y_{l}\partial y_{s}}+\gamma^{2}_{2}\frac{\partial^{2} Q(y_{0})}{\partial y_{l}\partial y_{s}}\Big)b_j^s=o(1),
$$
which implies that
$$
b_j^s=0,\,\,\,s=1,2,3.
$$

Step 3.  We are in a position to conclude Proposition \ref{lema.1}.

From  step 1 and step 2, we know that for any fixed $R>0$
$$
|(\xi_{1\epsilon_{n}},\xi_{2\epsilon_{n}})|\leq Ce^{-\frac{\theta R}{4}},\,\,\,\,\R^{3}\setminus B_{R\epsilon_{n}}(\eta^{j})
$$
and
$$
\xi_{i\epsilon_{n}}^{j}\rightarrow 0,\,\,\,\text{in}\,\,\,B_{R}(0),\,\,\,i=1,2,
$$
i.e.
$$
\xi_{i\epsilon_{n}}\rightarrow 0,\,\,\,\text{in}\,\,\,B_{R\epsilon_{n}}(\eta^{j}),\,\,\,i=1,2.
$$
Therefore, for any $x\in \R^{3}$ and sufficiently large $R>0,$ we can conclude that $\|(\xi_{1\epsilon_{n}},\xi_{2\epsilon_{n}})\|=o(1)$ which contradicts with $\|(\xi_{1\epsilon_{n}},\xi_{2\epsilon_{n}})\|_{L^{\infty}(\R^{3})}=1.$
\end{proof}

\medskip
\section{Existence of  dichotomous peak-solutions}
In this section, we construct dichotomous synchronized  and segregated solutions respectively in the following subsections.

Letting
\begin{eqnarray}\label{eqs3.1}
I(u,v)
&=&\frac{1}{2}\ds\int_{\R^3}\big(\epsilon^{2}|\nabla u|^2+P(x)u^2+\epsilon^{2}|\nabla v|^2+Q(x)v^2\big)
\cr
&&-\frac{1}{4}\ds\int_{\R^3}(\mu_1 u^4+\mu_2 v^4)
-\ds\frac{1}{2}\int_{\R^3}\beta u^2v^2 ,
\end{eqnarray}
then $I\in C^2$ and its critical points are solutions of \eqref{eq0}.


\subsection{Dichotomy for Synchronized solutions}
Let $D_k:=\Big[\big(\frac{m}{2\pi}-\delta\big)k\ln k,\big(\frac{m}{2\pi}+\delta\big)k\ln k\Big]$
with $\delta>0$ a small constant, where $m=\min\{m_1,m_2\}$.
Denote $W_{\epsilon,x^j}:=W_\epsilon(x-x^j)=W\big(\frac{x-x^j}{\epsilon}\big)$ and
\begin{align*}
Y_j=\frac{\partial U_{\epsilon,x^j}}{\partial r}, Z_j=\frac{\partial V_{\epsilon,x^j}}{\partial r},\,j=1,2\cdots,k,
\end{align*}
\begin{equation}\label{eqs03.2}
\begin{array}{ll}
E_k=\Bigl\{(u,v)\in H_{\epsilon,P}\times H_{\epsilon,Q},
 \ds\sum\limits_{j=1}^k\ds\int_{\R^3}W_{\epsilon,x^j}^{2}\Bigl(Y_ju+Z_jv\Bigr)=0\Bigr\}.
 \end{array}
\end{equation}

Define
$$
 J(\varphi_{k},\psi_{k})=I(S_{\epsilon,k}(x)+\varphi_{k}(x),T_{\epsilon,k}(x)+\psi_{k}(x)),\ \ \ (\varphi_{k},\psi_{k})\in E_k,
$$
where
$S_{\epsilon,k}(x):=u_{\epsilon}(x)+\sum\limits_{j=1}^{k}U_{\epsilon,x^j}(x)$ and
 $T_{\epsilon,k}(x):=v_{\epsilon}(x)+\sum\limits_{j=1}^{k}V_{\epsilon,x^{j}}(x).
 $
Then we can expand $J(\varphi_{k},\psi_{k})$ as
\begin{eqnarray}\label{eqs03.3}
J(\varphi_{k},\psi_{k})=J(0,0)+\ell(\varphi_{k},\psi_{k})+\frac{1}{2}\langle L(\varphi_{k},\psi_{k}),(\varphi_{k},\psi_{k})\rangle-R(\varphi_{k},\psi_{k}),\ \ (\varphi_{k},\psi_{k})\in E_k,
\end{eqnarray}
where
\begin{eqnarray}\label{eqs2.4}
\ell(\varphi_{k},\psi_{k})
&=&-\mu_1\int_{\R^3}\Bigl(S_{\epsilon,k}^3-u^{3}_{\epsilon}-\sum\limits_{j=1}^kU_{\epsilon,x^j}^3\Bigr)\varphi_{k} -\mu_2\int_{\R^3}\Bigl(T_{\epsilon,k}^3-v^{3}_{\epsilon}-\sum\limits_{j=1}^kV_{\epsilon,x^j}^3\Bigr)\psi_{k}
\cr
&&
+\int_{\R^3}\Big(\sum_{j=1}^{k}(P(x)-1)U_{\epsilon,x^{j}}\varphi_{k}+\sum_{j=1}^{k}(Q(x)-1)V_{\epsilon,x^{j}}\psi_{k}\Big)
\cr
&&
-\beta\int_{\R^3}\Bigl(S_{\epsilon,k}^2T_{\epsilon,k}-u_{\epsilon}^2v_{\epsilon}-\sum\limits_{j=1}^kU_{\epsilon, x^j}V_{\epsilon, x^j}^2\Bigr)\psi_{k} \cr
&&-\beta\int_{\R^3}\Bigl(T_{\epsilon,k}^2S_{\epsilon,k}-v_{\epsilon}^2u_{\epsilon}-\sum\limits_{j=1}^kU_{\epsilon, x^j}^2V_{\epsilon, x^j}\Bigr)\varphi_{k},
\end{eqnarray}

\begin{eqnarray}
\langle L(\varphi_{k},\psi_{k}),(\varphi_{k},\psi_{k})\rangle
&=&\int_{\R^3}\Bigl(\epsilon^{2}|\nabla\varphi_{k}|^2+P(x)\varphi_{k}^2-3\mu_1 S_{\epsilon, k}^2\varphi_{k}^{2}\Bigr)
+\int_{\R^3}\Bigl(\epsilon^{2}|\nabla\psi_{k}|^2+Q(x)\psi_{k}^2-3\mu_2 T_{\epsilon, k}^2
\psi_{k}^2\Bigr)
\cr
&&
-\beta\int_{\R^3}\Bigl(S_{\epsilon, k}^2\psi_{k}^2+4S_{\epsilon, k} T_{\epsilon, k}\varphi_{k}\psi_{k}
+T_{\epsilon, k}^2\varphi_{k}^2\Bigr),
\end{eqnarray}
is the quadratic term, and $R(\varphi,\psi,\xi)$ is higher order term, defined by
\begin{eqnarray}\label{wqf2.5}
R(\varphi_{k},\psi_{k})
&=&
-\int_{\R^3}\Bigl(\mu_1 S_{\epsilon, k}\varphi_{k}^3+\mu_2 T_{\epsilon, k}\psi_{k}^3+\frac{\mu_1}{4}\varphi_{k}^4+\frac{\mu_2}{4}\psi_{k}^4\Bigr)
\cr
&&
-\beta\int_{\R^3}\big(S_{\epsilon, k}\varphi_{k}\psi_{k}^2+T_{\epsilon, k}\varphi_{k}^2\psi_{k}+\frac{1}{2}\varphi_{k}^2\psi_{k}^2\big).
\end{eqnarray}
\begin{lem} \label{lm2.1}
There is a constant $C>0$, independent of $k$, such that for any $r\in D_k$,
$$
L(\varphi_{k},\psi_{k})\leq C\|(\varphi_{k},\psi_{k})\|,\  (\varphi_{k},\psi_{k})\in E_k.
$$
\end{lem}

\begin{lem}\label{lm2.2}
There exist a constant $ \varrho>0$, such that for any $ r\in D_k$,
$$
L(\varphi_{k},\psi_{k})\geq \varrho\|(\varphi_{k},\psi_{k})\|,\  \ (\varphi_{k},\psi_{k})\in E_k.
$$
\end{lem}
\begin{proof}
We argue by contradiction. Assume that there exist $k\rightarrow+\infty$, $(\varphi_k,\psi_k)\in E_k$ satisfying $\|L(\varphi_k,\psi_k)\|=o(1)\|(\varphi_k,\psi_k)\|^2$.
Without loss of generality, we may assume $\|(\varphi_k,\psi_k)\|^2=k$, we have
\begin{align*}
\langle L(\varphi_k,\psi_k),(g,h)\rangle=o(1)\|(\varphi_k,\psi_k)\|\|(g,h)\|.
\end{align*}
Then
\begin{align*}
&\int_{\R^3}(\epsilon^2\nabla\varphi_k\nabla g+P(x)\varphi_k g-3\mu_1S_{\epsilon, k}^2\varphi_k g)+\int_{\R^3}(\epsilon^2\nabla\psi_k\nabla h+Q(x)\psi_k h-3\mu_2T_{\epsilon, k}^2\psi_k h)\cr
&-\beta\int_{\R^3}(S_{\epsilon, k}^2\psi_kh+T_{\epsilon, k}^2\varphi_kg+2S_{\epsilon, k}T_{\epsilon, k}\varphi_kh+2S_{\epsilon, k}T_{\epsilon, k}\psi_kg)
=o(\sqrt{k})\|(g,h)\|.
\end{align*}
Taking $(g,h)=(\varphi_k,\psi_k)$, we have
\begin{align}\label{eq3.7}
o(k)=&\int_{\R^3}(\epsilon^2|\nabla\varphi_k|^2+P(x)\varphi_k^2-3\mu_1S_{\epsilon, k}^2\varphi_k^2)+\int_{\R^3}(\epsilon^2|\nabla\psi_k|^2+Q(x)\psi_k^2-3\mu_2T_{\epsilon, k}^2\psi_k^2)\cr
&-\beta\int_{\R^3}\big((S_{\epsilon, k}^2\psi_k^2+T_{\epsilon, k}^2\varphi_k^2+4S_{\epsilon, k}T_{\epsilon, k}\varphi_k\psi_k)\big).
\end{align}
 Set
\begin{align*}
\bar{\varphi}_k=\varphi_k(x+x^1),\,\,\,\bar{\psi}_k=\psi_k(x+x^1).
\end{align*}
For any fixed $R>0$, $B_R(x_1)\subset\Omega_1$ since $|x^j-x^1|\geq r\sin\frac{\pi}{k}\geq C\ln k$ for $j=1,\cdots,k$.
We have
\begin{align*}
\int_{B_R(0)}\big(\epsilon^2|\nabla\bar{\varphi}_k|^2+P(x)\bar{\varphi}_k^2+\epsilon^2|\nabla\bar{\psi}_k|^2+Q(x)\bar{\psi}_k^2\big)\leq1.
\end{align*}
Up to a subsequence, we may assume that there exist $(\varphi,\psi)\in H^1(\R^3)\times H^1(\R^3)$ such that
\begin{align}\label{eq1.1}
&(\bar{\varphi}_k,\bar{\psi}_k)\rightharpoonup(\varphi,\psi)\,\,\, \hbox{weakly in }\, H_{loc}^1(\R^3)\times H_{loc}^1(\R^3);\cr
&(\bar{\varphi}_k,\bar{\psi}_k)\rightarrow(\varphi,\psi)\,\,\,\hbox{strongly in}\,L_{loc}^q(\R^3)\times L_{loc}^q(\R^3),q\in[2,6).
\end{align}
Moreover, $\varphi,\psi$ are even in $x_e(e=2,3)$ and satisfy
\begin{align}\label{eq3.8}
\int_{\R^3}W_\epsilon^2(\frac{\partial U_\epsilon}{\partial x_1}\varphi+\frac{\partial V_\epsilon}{\partial x_1}\psi)=0.
\end{align}

Now, we claim that $(\varphi,\psi)$ satisfies
\begin{eqnarray}\begin{cases}\label{05}
-\epsilon^2\Delta\varphi+\varphi-3\mu_1U_\epsilon^2\varphi-\beta V_\epsilon^2\varphi-2\beta U_\epsilon V_\epsilon\psi=0,\cr
-\epsilon^2\Delta\psi+\psi-3\mu_2V_\epsilon^2\psi-\beta V_\epsilon^2\psi-2\beta U_\epsilon V_\epsilon\varphi=0.
\end{cases}
\end{eqnarray}
For any $R>0$, letting $(g,h)\in C_0^\infty(B_R(0))\times  C_0^\infty(B_R(0))$ and $(g,h)$ be even in $x_e,e=2,3$, $(g_1(x),h_1(x))=(g(x-x^1),h(x-x^1)\in C_0^\infty(B_R(x_1))\times C_0^\infty(B_R(x_1))\subset\Omega_1\times\Omega_1$, we have
\begin{align}\label{eqs2.9}
\int_{\Omega_1}(\epsilon^2\nabla\varphi_k\nabla g_1+P(x)\varphi_kg_1-3\mu_1S_{\epsilon, k}^2\varphi_kg_1)\rightarrow\int_{\R^3}(\epsilon^2\nabla\varphi\nabla g+\varphi g-3\mu_1U_\epsilon^2\varphi g),
\end{align}
\begin{align}\label{eqs2.10}
\int_{\Omega_1}(\epsilon^2\nabla\psi_k\nabla h_1+Q(x)\psi_kh_1-3\mu_2T_{\epsilon, k}^2\psi_kh_1)\rightarrow\int_{\R^3}(\epsilon^2\nabla\psi\nabla h+\psi h-3\mu_2V_\epsilon^2\psi h),
\end{align}
and
\begin{align}\label{eqs2.11}
&\int_{\Omega_1}(S_{\epsilon, k}^2\psi_kh_1+T_{\epsilon, k}^2\varphi_kg_1+2S_{\epsilon, k}T_{\epsilon, k}\varphi_kh_1+2S_{\epsilon, k}T_{\epsilon, k}\psi_kg_1)\cr
&\rightarrow\int_{\R^3}(U_\epsilon^2\psi h+V_\epsilon^2\varphi g+2U_\epsilon V_\epsilon\varphi h+2U_\epsilon V_\epsilon\psi g).
\end{align}
Insertting \eqref{eqs2.9}-\eqref{eqs2.11} into \eqref{eq3.7}, we see
\begin{align}\label{eqs2.12}
&\int_{\R^3}\big(\epsilon^2\nabla\varphi\nabla g+\varphi g-3\mu_1U_\epsilon^2\varphi g\big)+\int_{\R^3}\big(\epsilon^2\nabla\psi\nabla h+\psi h-3\mu_2V_\epsilon^2\psi h\big)\cr
&-\beta\int_{\R^3}\big(U_\epsilon^2\psi h+V_\epsilon^2\varphi g+2U_\epsilon V_\epsilon\varphi h+2U_\epsilon V_\epsilon\psi g\big)=0.
\end{align}
By the density of $C_0^\infty(B_R(0))\times C_0^\infty(B_R(0))$ in $H^1(\R^3)\times H^1(\R^3)$,  we have proved \eqref{05}.

Since we work in the space of functions which are even in $x_2,x_3$, the kernel of $(U_\epsilon,V_\epsilon)$ is given by one dimensional $(\frac{\partial W_\epsilon}{\partial x_1},\frac{\partial W_\epsilon}{\partial x_1})$. Thus, we see $(\varphi,\psi)=c(\frac{\partial U_\epsilon}{\partial x_1},\frac{\partial V_\epsilon}{\partial x_1})$ for some constant $c$, which from \eqref{eq3.8} implies that $(\varphi,\psi)=(0,0)$.

As a result, there holds
\begin{align*}
\int_{B_R(x^1)}\varphi_k^2+\psi_k^2=o(1),\forall\,\,R>0.
\end{align*}
On the other hand, from $\|(\varphi_k,\psi_k)\|^2=k$, we know $\|\frac{1}{\sqrt{k}}(\varphi_k,\psi_k)\|=1$, so there exists $(\varphi^*,\psi^*)\in H^1(\R^3)\times H^1(\R^3)$ such that
\begin{align}\label{eq3}
&\frac{1}{\sqrt{k}}\varphi_k\rightharpoonup\varphi^*,\,\,\, \frac{1}{\sqrt{k}}\psi_k\rightharpoonup\psi^* \,\,\hbox{weakly in}\,\,H^1(\R^3);\cr
&\frac{1}{\sqrt{k}}\varphi_k\rightarrow\varphi^*,\,\,\,\,\frac{1}{\sqrt{k}}\psi_k\rightarrow\psi^*\,\,\hbox{strongly in}\,\,\,L_{loc}^q(\R^3),q\in [2,6).
\end{align}
Taking $(f,g)\in C_0^\infty(B_R(0))\times C_0^\infty(B_R(0))$ with $\|(f,g)\|=1$ and letting $k\rightarrow+\infty$, it holds
\begin{align*}
\int_{\R^3}\epsilon^2\nabla(\frac{1}{\sqrt{k}}\varphi_k)\nabla f+&P(x)\frac{1}{\sqrt{k}}\varphi_kf-3\mu_1S_{\epsilon, k}^2\frac{1}{\sqrt{k}}\varphi_kf\cr
&\rightarrow\int_{\R^3}\epsilon^2\nabla \varphi^*\nabla f+P(x)\varphi^*f-3\mu_1u^2_{\epsilon}\varphi^*f,
\end{align*}
\begin{align*}
\int_{\R^3}\epsilon^2\nabla(\frac{1}{\sqrt{k}}\psi_k)\nabla g+&Q(x)\frac{1}{\sqrt{k}}\psi_kg-3\mu_2T_{\epsilon, k}^2\frac{1}{\sqrt{k}}\psi_kg\cr
&\rightarrow\int_{\R^3}\epsilon^2\nabla \psi^*\nabla g+Q(x)\psi^*g-3\mu_2v^2_{\epsilon}\psi^*g,
\end{align*}
and
\begin{align*}
&\int_{\R^3}S_{\epsilon, k}^2(\frac{1}{\sqrt{k}}\psi_k)g+T_{\epsilon, k}^2(\frac{1}{\sqrt{k}}\varphi_k)f+2S_{\epsilon, k}T_{\epsilon, k}(\frac{1}{\sqrt{k}}\varphi_k)g+2S_{\epsilon, k}T_{\epsilon, k}(\frac{1}{\sqrt{k}}\psi_k)f\cr
&\rightarrow\int_{\R^3}u_\epsilon^2\psi^*g+v_\varepsilon^2\varphi^*f+2u_\epsilon v_\epsilon\varphi^*g+2u_\epsilon v_\epsilon\psi^*f.
\end{align*}
Then $(\varphi^*,\psi^*)$ satisfies
\begin{eqnarray}\begin{cases}\label{050}
-\epsilon^2\Delta\varphi^*+P(x)\varphi^*=3\mu_1u_\epsilon^2\varphi^*+\beta(2u_\epsilon v_\epsilon\psi^*+v_\epsilon^2\varphi^*),\,\,\,x\in \R^{3},\cr
-\epsilon^2\Delta\psi^*+Q(x)\psi^*=3\mu_2v_\epsilon^2\psi^*+\beta(2u_\epsilon v_\epsilon\varphi^*+u_\epsilon^2\psi^*),\,\,\,\,\,x\in \R^{3}.
\end{cases}
\end{eqnarray}

It follows from Proposition \ref{lema.1} that $(u_\epsilon,v_\epsilon)$ is non-degenerate. Then $(\varphi^*,\psi^*)=0$ and we get
for $2\leq q<6$,
\begin{align*}
\Big(\int_{B_R(0)}\varphi_k^qdx\Big)^{\frac{1}{q}}=o(\sqrt{k}),\,\,\,\,\,\,\,\,\Big(\int_{B_R(0)}\psi_k^qdx\Big)^{\frac{1}{q}}=o(\sqrt{k}).
\end{align*}
Using Lemma \ref{lemb.1}, we have
\begin{align*}
o(k)=&\int_{\R^3}\epsilon^2|\nabla\varphi_k|^2+P(x)\varphi_k^2-3\mu_1S_{\epsilon, k}^2\varphi_k^2+\int_{\R^3}\epsilon^2|\nabla\psi_k|^2+Q(x)\psi_k^2-3\mu_2T_{\epsilon, k}^2\psi_k^2\cr
&-\beta\int_{\R^3}(S_{\epsilon, k}^2\psi_k^2+T_{\epsilon, k}^2\varphi_k^2+4S_{\epsilon, k}T_{\epsilon,k}\varphi_k\psi_k)\cr
=&k-3\mu_1\int_{\R^3}(u_{\epsilon}^2+U_{\epsilon, r}^2)\varphi_k^2-3\mu_2C\int_{\R^3}(v_{\epsilon}^2+U_{\epsilon,r}^2)\psi_k^2
-\beta\int_{\R^3}\big((\bar{S}_{\epsilon, k}^2\psi_k^2+\bar{T}_{\epsilon, k}^2\varphi_k^2)\big)+o(k)\cr
\geq&k-(o(1)+O(e^{-\frac{R}{\epsilon}}))k,
\end{align*}
which  is impossible for large $R$ and $k$. Consequently, we complete the proof.
\end{proof}

\begin{lem}\label{lm2.4}
There exists a constant $C>0$, independent of $k$ such that
$$
\|R^{(i)}(\varphi,\psi)\|\leq C\|(\varphi,\psi)\|^{3-i},\ i=0,1,2.
$$
\end{lem}
\begin{proof}
Since these estimates can be obtained by direct computations, we can refer to \cite{PW13}.
\end{proof}

\begin{lem}\label{lm2.5}
There exists a constant $C>0$, independent of $k$, such that
$$
\|\ell\|\leq C\Bigl(\frac{k}{r^{m_1}}+\frac{k}{r^{m_2}}+\frac{k}{r}e^{-\frac{r}{\epsilon}}+\frac{k}{r}\sum\limits_{i\neq j}e^{-\frac{|x^i-x^j|}{\epsilon}}\Bigr).
$$
\end{lem}
\begin{proof}
From \eqref{eqs2.4}, we have
\begin{eqnarray}\label{eqs2.15}
\ell(\varphi_{k},\psi_{k})
&=&
\int_{\R^3}\Big(\sum_{j=1}^{k}(P(x)-1)U_{\epsilon,x^{j}}\varphi_{k}+\sum_{j=1}^{k}(Q(x)-1)V_{\epsilon,x^{j}}\psi_{k}\Big)
\cr
&&-\mu_1\int_{\R^3}\Bigl(S_{\epsilon,k}^3-u^{3}_{\epsilon}-\sum\limits_{j=1}^kU_{\epsilon,x^j}^3\Bigr)\varphi_{k} -\mu_2\int_{\R^3}\Bigl(T_{\epsilon,k}^3-v^{3}_{\epsilon}-\sum\limits_{j=1}^kV_{\epsilon,x^j}^3\Bigr)\psi_{k}
\cr
&&
-\beta\int_{\R^3}\Bigl(S_{\epsilon,k}^2T_{\epsilon,k}-u_{\epsilon}^2v_{\epsilon}-\sum\limits_{j=1}^kU_{\epsilon, x^j}^2V_{\epsilon, x^j}\Bigr)\psi_{k} \cr
&&-\beta\int_{\R^3}\Bigl(T_{\epsilon,k}^2S_{\epsilon,k}-v_{\epsilon}^2u_{\epsilon}-\sum\limits_{j=1}^kU_{\epsilon, x^j}V_{\epsilon, x^j}^2\Bigr)\varphi_{k}\cr
&=:&K_1+K_2+K_3+K_4.
\end{eqnarray}
By the symmetry, we have
\begin{align}\label{2.1}
&\int_{\R^3}\Big(\sum_{j=1}^{k}(P(x)-1)U_{\epsilon,x^{j}}\varphi_{k}=k\int_{\R^3}(P(x)-1)U_{\epsilon,x^{1}}\varphi_{k}\cr
=&k\int_{\R^3}((P(x-x^1)-1))U_\epsilon\varphi_{k}(x-x^1)
\leq C\frac{k}{r^{m_1}}\|(\varphi_k, \psi_{k})\|+O(e^{-\frac{k\pi}{\epsilon}}).
\end{align}
Similarly, we have
\begin{align*}
\int_{\R^3}\Big(\sum_{j=1}^{k}(Q(x)-1)V_{\epsilon,x^{j}}\psi_{k}\leq C\frac{k}{r^{m_2}}\|(\varphi_k, \psi_{k})\|+O(e^{-\frac{k\pi}{\epsilon}}).
\end{align*}
Then
\begin{align}\label{eqs2.10}
K_1=\int_{\R^3}\Big(\sum_{j=1}^{k}(P(x)-1)U_{\epsilon,x^{j}}\varphi_{k}+\sum_{j=1}^{k}(Q(x)-1)V_{\epsilon,x^{j}}\psi_{k}\Big)\leq C(\frac{k}{r^{m_1}}+\frac{k}{r^{m_2}})\|(\varphi_k,\psi_k)\|.
\end{align}
On the other hand, we obtain
\begin{align}\label{2.3}
&\int_{\R^3}\Bigl(S_{\epsilon,k}^3-u^{3}_{\epsilon}-\sum\limits_{j=1}^kU_{\epsilon,x^j}^3\Bigr)\varphi_{k}
=\int_{\R^3}\Big((u_{\epsilon}+\sum\limits_{j=1}^kU_{\epsilon,x^j})^3-\sum\limits_{j=1}^kU_{\epsilon, x^j}^3-u_{\epsilon}^3\Big)\varphi_k\cr
=&3\int_{\R^3}\Big(u_{\epsilon}^2\sum\limits_{j=1}^kU_{\epsilon,x^j}+\big(\sum\limits_{j=1}^kU_{\epsilon,x^j}\big)^2u_\epsilon+\sum\limits_{i\neq j}U^2_{\epsilon, x^i}U_{\epsilon, x^j}\Big)\varphi_k\\
=:&3(K_{21}+K_{22}+K_{23}).
\end{align}
For $\tau$ being an integer, we write
\begin{align}\label{2.4}
&\int_{\R^3}u_{\epsilon}^2\sum\limits_{j=1}^kU_{\epsilon, x^j}\varphi_k=k\int_{\R^3}u_{\epsilon}^2U_{\epsilon, x^1}\varphi_k\cr
=&k\int_{B_\tau(x^1)}u_\epsilon^2U_{\epsilon, x^1}\varphi_k+k\int_{\R^3\setminus B_\tau(x^1)}u_{\epsilon}^2U_{\epsilon, x^1}\varphi_k=:k(K_{211}+K_{212}),
\end{align}

\begin{align}\label{2.5}
K_{212}=\int_{\R^3\setminus B_{\tau}(x^1)}u_\epsilon^2U_{\epsilon, x^1}\varphi_k \leq& \int_{B_R(0)}u_{\epsilon}U_{\epsilon, x^1}\varphi_k+\int_{\R^3\setminus(B_\tau(x^1)\bigcup B_R(0))}u_{\epsilon}^2U_{\epsilon, x^1}\varphi_k\cr
\leq& (\int_{B_R(0)}u_{\epsilon}^6)^{\frac{1}{3}}(\int_{\R^3\setminus B_{\tau}(x^1)}U_{\epsilon, x^1}^3)^{\frac{1}{3}}(\int_{\R^3\setminus B_{\tau}(x^1)}\varphi_k^3)^{\frac{1}{3}}\cr
\leq&C\frac{1}{r}e^{-\frac{r}{\epsilon}}\|(\varphi_k,\psi_k)\|,
\end{align}
and
\begin{align}\label{2.6}
K_{211}=&\int_{B_\tau(x^1)}u_{\epsilon}^2U_{\epsilon, x^1}\varphi_k \leq\int_{\R^3\setminus B_\tau(0)}u_{\epsilon, k}^2U_{\epsilon, x^1}\varphi_k\cr
\leq&(\int_{\R^3\setminus B_\tau(0)}u_\epsilon^6)^{\frac{1}{3}}(\int_{B_{\tau}(x^1)}U_{\epsilon, x^1}^3)^{\frac{1}{3}}(\int_{B_{\tau}(x^1)}\varphi_k^3)^{\frac{1}{3}}\cr
\leq&C\frac{1}{\sqrt{k}r}e^{-\frac{r}{\epsilon}}\|(\varphi_k,\psi_k)\|.
\end{align}
Inserting \eqref{2.5} and \eqref{2.6} into \eqref{2.4}, we obtain
\begin{align}\label{2.70}
K_{21}=\int_{\R^3}u_{\epsilon}^2\sum\limits_{j=1}^kU_{\epsilon, x^j}\varphi_k\leq (\frac{k}{r}e^{-\frac{r}{\epsilon}}+\frac{\sqrt{k}}{r}e^{-\frac{r}{\epsilon}})\|(\varphi_k,\psi_k)\|.
\end{align}
Similarly, we obtain
\begin{align}\label{2.700}
K_{22}=\int_{\R^3}\sum\limits_{j=1}^kU_{\epsilon, x^j}^2u_\epsilon\varphi_k\leq (\frac{k}{r}e^{-\frac{r}{\epsilon}}+\frac{\sqrt{k}}{r}e^{-\frac{r}{\epsilon}})\|(\varphi_k,\psi_k)\|,
\end{align}
and
\begin{align}\label{2.71}
K_{23}=\int_{\R^3}\sum\limits_{i\neq j}U_{\epsilon, x^i}U_{\epsilon, x^j}\varphi_k\leq C\frac{k}{r}\sum\limits_{i\neq j}e^{-\frac{|x^i-x^j|}{\epsilon}}\|(\varphi_k,\psi_k)\|.
\end{align}
Combining \eqref{2.70} and \eqref{2.71}, we obtain
\begin{align*}
K_2=&-\mu_1\int_{\R^3}\Bigl(S_{\epsilon,k}^3-u^{3}_{\epsilon}-\sum\limits_{j=1}^kU_{\epsilon,x^j}^3\Bigr)\varphi_{k} -\mu_2\int_{\R^3}\Bigl(T_{\epsilon,k}^3-v^{3}_{\epsilon}-\sum\limits_{j=1}^kV_{\epsilon,y^j}^3\Bigr)\psi_{k}\cr
\leq& C(\frac{k}{r}e^{-\frac{r}{\epsilon}}+C\frac{\sqrt{k}}{r}e^{-\frac{r}{\epsilon}}+\frac{k}{r}\sum\limits_{i\neq j}^ke^{-\frac{|x^i-x^j|}{\epsilon}})\|(\varphi_k,\psi_k)\|.
\end{align*}
Also, by direct computations, we have
\begin{align}\label{eqs2.51}
K_3=\int_{\R^3}\Big(S_{\epsilon, k}^2T_{\epsilon, k}-u^2_\epsilon v_\epsilon-\sum\limits_{i\neq j}U_{\epsilon, x^i}V_{\epsilon, x^j}\Big)=O\Big(\int_{\R^3}u_\epsilon v_\epsilon U_{\epsilon, r}+u_\epsilon U_{\epsilon, r}V_{\epsilon, r}\Big)=O\big(\frac{k}{r}e^{-\frac{r}{\epsilon}}\big),
\end{align}
and
\begin{align}\label{eqs2.52}
K_4=\int_{\R^3}\Big(S_{\epsilon, k}T_{\epsilon, k}^2-u_\epsilon v_\epsilon^2-\sum\limits_{i\neq j}U_{\epsilon, x^i}V_{\epsilon, x^j}^2\Big)=O\Big(\int_{\R^3}u_\epsilon v_\epsilon V_{\epsilon r}+v_\epsilon U_{\epsilon, r}V_{\epsilon, r}\Big)=O\big(\frac{k}{r}e^{-\frac{r}{\epsilon}}\big).
\end{align}
From \eqref{eqs2.15} to \eqref{eqs2.52}, we complete the results.
\end{proof}

\begin{prop}\label{pro2.6}
There is an integer $k_0>0$ such that for each $k\geq k_0$,
there is a $C^1$ map from $D_k$ to $E:(\varphi,\psi)=(\varphi(r),\psi(r))$, $r=|x^1|$ satisfying $(\varphi,\psi)\in E_k$, and
 $$
 \frac{\partial J(\varphi,\psi)}{\partial(\varphi,\psi)}\Big|_{E_k}=0.
 $$
Moreover, there is a constant $C$, such that
$$
\|(\varphi,\psi)\|\leq C\Bigl(\frac{k}{r^{m_1}}+\frac{k}{r^{m_2}}+\frac{k}{r}e^{-\frac{r}{\epsilon}}+\frac{k}{r}\sum\limits_{i\neq j}e^{-\frac{|x^i-x^j|}{\epsilon}}\Bigr).
$$
\end{prop}
\begin{proof}
We know that $\ell(\varphi,\psi)$ is a bounded linear functional in $E_k$, which implies that  there is an $\ell_\epsilon\in E_k$, such that
$$
\ell(\varphi,\psi) =\langle \ell_\epsilon,(\varphi,\psi)\rangle.
$$
Thus finding a critical point for $J(\varphi,\psi)$ is equivalent to solving
$$
\ell_\epsilon+L(\varphi,\psi)+R^{'}(\varphi,\psi)=0,
$$
which is also equivalent to
$$
(\varphi,\psi)=A(\varphi,\psi)=:-L^{-1}\ell_\epsilon-L^{-1}R^{'}(\varphi,\psi),
$$
since $L$ is invertible in $E_k$.

Set
\[
M=\Bigl\{ (\varphi,\psi)\in E:\,\,\|(\varphi,\psi)\|\leq
C\Bigl(\frac{k}{r^{m_1}}+\frac{k}{r^{m_2}}+\frac{k}{r}e^{-\frac{r}{\epsilon}}+\frac{k}{r}\sum\limits_{i\neq j}e^{-\frac{|x^i-x^j|}{\epsilon}}\Bigr)\Bigr\}.
\]

 Following Lemmas \ref{lm2.2}-\ref{lm2.5}, we can prove that  $A$ is a contraction map on $M$. Then by the contraction mapping theorem, we conclude that there exists $(\varphi,\psi)\in E_k$, such that $A(\varphi,\psi)=(\varphi,\psi)$. Moreover, there holds
$$
\|(\varphi,\psi)\|\leq  C\Bigl(\frac{k}{r^{m_1}}+\frac{k}{r^{m_2}}+\frac{k}{r}e^{-\frac{r}{\epsilon}}+\frac{k}{r}\sum\limits_{i\neq j}e^{-\frac{|x^i-x^j|}{\epsilon}}\Bigr).
$$
\end{proof}

Now we are ready to prove Theorem \ref{th3}.
\begin{proof}[\textbf{Proof of Theorem \ref{th3}}]
 We prove the theorem only for the case $m_1=m_2$, since the other case is similar.
If $m_1=m_2$, using from Lemmas~\ref{lm2.2}-\ref{lm2.5}, Propositions~\ref{pro2.6} and \ref{propB.1} imply that
\begin{eqnarray*}
F(r)
&=&
I(S_{\epsilon,k},T_{\epsilon,k})+\ell(\varphi_{k},\psi_{k})+\frac{1}{2}\langle L(\varphi_{k},\psi_{k}),(\varphi_{k},\psi_{k})\rangle-R(\varphi_{k},\psi_{k})\nonumber\\
&=&
I(S_{\epsilon,k},T_{\epsilon,k})+O\big(\|\ell\|\|(\varphi_k,\psi_k)\|+\|(\varphi_k,\psi_k)\|^2\big)\nonumber\\
&=&
I(S_{\epsilon,k},T_{\epsilon,k})+O\big(\frac{1}{r^{m-1-\sigma}}\big)\nonumber\\
&=&A+k(\frac{a_1\gamma_12}{r^m}+\frac{a_2\gamma_2^2}{r^m})B-C_\beta\frac{k^2}{r}e^{-\frac{r\pi}{\epsilon k}}+O(ke^{-\frac{\pi r}{\epsilon k}}).
\end{eqnarray*}
Consider the following maximization problem
\begin{equation}\label{e2.8}
\max\limits_{r\in D_k}F(r).
\end{equation}
Assume that \eqref{e2.8} is achieved by some $r^{*}$ in $D_k.$ We will prove that $r^{*}$ is an interior point of
$D_k$.
Note that $C_\beta>0$.
Define
\begin{align*}
g(t)=\frac{aB}{r^m}-C_\beta\frac{k}{r}e^{-\frac{2\pi r}{\epsilon k}}=\frac{aB}{t^mk^m}-C_\beta\frac{e^{-2\pi t}}{t}.
\end{align*}
Then
\begin{align*}
g'(t)=-\frac{aB}{t^{m+1}k^m}+\frac{2\pi C_\beta e^{-2\pi t}}{t}+\frac{C_\beta e^{-2\pi t}}{t^2}.
\end{align*}
It is easy to check that $g(t)$ has a maximum point of $t_k$ with
\begin{align*}
t_k=\big(\frac{m}{2\pi}+o(1)\big)\ln k.
\end{align*}
So the function $F(r)$ has a maximum point of
$$r=kt_k=\big(\frac{m}{2\pi}+o(1)\big)k\ln k.$$
Hence, it follows from the expression of $F(r)$ that the maximizer of $r_k$ is an interior point of $D_k$. We can use the same argument as \cite{PW13} to show the solutions are positive.
\end{proof}
\subsection{Dichotomy for Segregated solutions}
In this subsection, we construct dichotomous segregated peak-solutions of \eqref{eq0}.

Recall that
$$
 \bar{S}_{\epsilon,k}(y):=\sum_{j=1}^{k}W_{\epsilon,\mu_1,x^{j}}(y),~~~~~\bar{T}_{\epsilon,k}(y):=\sum_{j=1}^{k}W_{\epsilon,\mu_2,y^{j}}(y).
 $$
Now we define that
\begin{align*}
Y_j=\frac{\partial W_{\epsilon,\mu_1,x^j}}{\partial r},~~~ Z_j=\frac{\partial W_{\epsilon,\mu_2,y^j}}{\partial \rho},\,j=1,2\cdots,k,
\end{align*}
and
\begin{equation}
\begin{array}{ll}
\bar{E}_k=\Bigl\{(u,v)\in H_{\epsilon,P}\times H_{\epsilon,Q},
 \ds\sum\limits_{j=1}^k\ds\int_{\R^3}W_{\epsilon,\mu_1,x^j}^{2}Y_ju=0,\,\,\sum\limits_{j=1}^k\int_{\R^3}
 W_{\epsilon,\mu_2,y^j}^{2}Z_jvdx=0\Bigr\}.
 \end{array}
\end{equation}

Let
$$
 J(\bar{\varphi}_{k},\bar{\psi}_{k})=I(\bar{S}_{\epsilon,k}(y)+\bar{\varphi}_{k}(y),\bar{T}_{\epsilon,k}(y)+\bar{\psi}_{k}(y)),\ \ \ (\bar{\varphi}_{k},\bar{\psi}_{k})\in \bar{E}_k.
$$
Then we expand $J(\varphi_{k},\psi_{k})$ as
\begin{eqnarray}
J(\bar{\varphi}_{k},\bar{\psi}_{k})=\bar{J}(0,0)+\bar{\ell}(\bar{\varphi}_{k},\bar{\psi}_{k})+\frac{1}{2}\langle \bar{L}(\bar{\varphi}_{k},\bar{\psi}_{k}),(\bar{\varphi}_{k},\bar{\psi}_{k})\rangle-\bar{R}(\bar{\varphi}_{k},\bar{\psi}_{k}),\ \ (\bar{\varphi}_{k},\bar{\psi}_{k})\in \bar{E}_k,
\end{eqnarray}
where
\begin{eqnarray}
\bar{\ell}(\bar{\varphi}_{k},\bar{\psi}_{k})
&=&-\mu_1\int_{\R^3}\Bigl(\bar{S}_{\epsilon,k}^3-u^{3}_{\epsilon}-\sum\limits_{j=1}^kW_{\epsilon,\mu_1,x^j}^3\Bigr)\bar{\varphi}_{k} dx-\mu_2\int_{\R^3}\Bigl(\bar{T}_{\epsilon,k}^3-v^{3}_{\epsilon}-\sum\limits_{j=1}^kW_{\epsilon,\mu_2,y^j}^3\Bigr)\bar{\psi}_{k}  dx
\cr
&&
+\int_{\R^3}\Big(\sum_{j=1}^{k}(P(y)-1)W_{\epsilon,\mu_1,x^{j}}\bar{\varphi}_{k}+\sum_{j=1}^{k}(Q(y)-1)
W_{\epsilon,\mu_2,y^{j}}\bar{\psi}_{k}\Big)dx
\cr
&&
-\beta\int_{\R^3}\Bigl(\bar{S}_{\epsilon,k}^2\bar{T}_{\epsilon,k}-u_{\epsilon}^2v_{\epsilon}\Bigr)\bar{\psi}_{k}dx
-\beta\int_{\R^3}\Bigl(\bar{T}_{\epsilon,k}^2\bar{S}_{\epsilon,k}-v_{\epsilon}^2u_{\epsilon}\Bigr)\bar{\varphi}_{k}dx,
\end{eqnarray}
and $\bar{L}(\bar{\varphi}_{k},\bar{\psi}_{k})$ and $\bar{R}(\bar{\varphi}_{k},\bar{\psi}_{k})$ are exactly as $L(\varphi_{k},\psi_{k})$ and $R(\varphi_k,\psi_k)$ in section 2, but $U_{\epsilon, x^j}, V_{\epsilon, x^j},\varphi$ and $\psi$ being replaced by $W_{\epsilon,\mu_1, x^j}, W_{\epsilon,\mu_2, x^j},\bar{\varphi}$ and $\bar{\psi}$ respectively.

\begin{lem} \label{lm4.1}
There is a constant $C>0$, independent of $k$, such that for any $(r,\rho)\in D_k\times D_k$,
$$
\bar{L}(\varphi_{k},\psi_{k})\leq C\|(\varphi_{k},\psi_{k})\|,\  (\varphi_{k},\psi_{k})\in \bar{E}_k.
$$
\end{lem}

\begin{lem}\label{lm4.2}
There exist $\varepsilon>0$ small enough  and $\varrho>0$, such that for any $\beta<\varepsilon, (r,\rho)\in D_k\times D_k$, there holds
$$
\bar{L}(\varphi_{k},\psi_{k})\geq \varrho\|(\varphi_{k},\psi_{k})\|,\  \ (\varphi_{k},\psi_{k})\in \bar{E}_k.
$$
\end{lem}
\begin{proof}
We argue by contradiction. Assume that there exist $k\rightarrow+\infty$, $(\varphi_k,\psi_k)\in \bar{E}_k$ satisfying $\|\bar{L}(\varphi_k,\psi_k)\|=o(1)\|(\varphi_k,\psi_k)\|^2$.
Without loss of generality, we may assume $\|(\varphi_k,\psi_k)\|^2=k$, we have
\begin{align*}
\langle \bar{L}(\varphi_k,\psi_k),(g,h)\rangle=o(1)\|(\varphi_k,\psi_k)\|\|(g,h)\|.
\end{align*}
Taking $(g,h)=(\varphi_k,\psi_k)$, we have
\begin{align*}
o(k)=&\int_{\R^3}\epsilon^2|\nabla\varphi_k|^2+P(x)\varphi_k^2-3\mu_1\bar{S}_{\epsilon, k}^2\varphi_k^2+\int_{\R^3}\epsilon^2|\nabla\psi_k|^2+Q(x)\psi_k^2-3\mu_2\bar{T}_{\epsilon, k}^2\psi_k^2\cr
&-\beta\int_{\R^3}(\bar{S}_{\epsilon, k}^2\psi_k^2+\bar{T}_{\epsilon, k}^2\varphi_k^2+4\bar{S}_{\epsilon, k}\bar{T}_{\epsilon, k}\varphi_k\psi_k).
\end{align*}
For any fixed $R>0$, there is large enough $k>0$ such that $B_R(x_1)\subset\Omega_1$. Set
\begin{align*}
\tilde{\varphi}_k=\varphi_k(x+x^1),\,\,\,\tilde{\psi}_k=\psi_k(x+y^1).
\end{align*}
We have
\begin{align*}
\int_{B_R(0)}\epsilon^2|\nabla \tilde{\varphi}_k|^2+P(x)\tilde{\varphi}_k^2+\epsilon^2|\nabla\tilde{\psi}_k|^2+Q(x)\tilde{\psi}_k^2\leq1.
\end{align*}
Up to a subsequence, we may assume there exist $(\varphi,\psi)\in H^1(\R^3)\times H^1(\R^3)$ such that
\begin{align}\label{eq1.1}
&(\tilde{\varphi}_k,\tilde{\psi}_k)\rightharpoonup(\varphi,\psi)\,\,\, \hbox{weakly in }\, H_{loc}^1(\R^3)\times H_{loc}^1(\R^3),\cr
&(\tilde{\varphi}_k,\tilde{\psi}_k)\rightarrow(\varphi,\psi)\,\,\,\hbox{strongly in}\,L_{loc}^q(\R^3)\times L_{loc}^q(\R^3),q\in[2,6).
\end{align}

Moreover, $\varphi,\psi$ are even in $x_2,x_3$ and satisfy
\begin{align}\label{eq2}
\int_{\R^3}U_\epsilon^2\frac{\partial U_{\epsilon,x^1}}{\partial r}\varphi=0,\,\int_{\R^3}U_\epsilon^2\frac{\partial U_{\epsilon,y^1}}{\partial \rho}\psi=0,
\end{align}
where $U_\epsilon=U(\frac{x}{\epsilon})$ and $U_{\epsilon,z}=U(\frac{x-z}{\epsilon})$ for some $z\in \R^3$.
Let $g\in C_0^\infty(B_R(0))$ and be even in $x_h,h=2,3$. Define $g_1(x)=g_1(x-x^1)\in C_0^\infty(B_R(x^1))$. Then choosing $(g,h)=(g_1,0)$ and proceeding as we did in Lemma \ref{lm2.2}, we can see that $\varphi$ satisfies $-\Delta\varphi+\varphi=3\mu_1U^2\varphi$.

Moreover, in view of the non-degeneracy of $U_\epsilon$, it holds $\varphi=0$.
By \eqref{eq1.1}, we see
\begin{align*}
o(1)=\int_{B_R(0)}\tilde{\varphi}_k^qdx=\int_{B_R(x^1)}\varphi_k^qdx.
\end{align*}
Using the same argument on $\tilde{\Omega}$, we can prove that as $k\rightarrow+\infty$,
\begin{align}\label{3.3}
\tilde{\psi}_k\rightarrow0,\,\,\hbox{weakly in }\, H_{loc}^1(\R^3),\,\,\,\tilde{\psi}_k\rightarrow0,\,\,\,\hbox{strongly in} \,L_{loc}^2(\R^3).
\end{align}
As a result, it holds
\begin{align*}
\int_{B_R(x^1)}\varphi^2=o(1),\,\,\,\,\,\,\int_{B_R(y^1)}\psi^2=o(1), \forall\,\,\,R>0.
\end{align*}
Similar as Lemma \ref{lm2.2}, we get
\begin{align*}
\Big(\int_{B_R(0)}\varphi_k^qdx\Big)^{\frac{1}{q}}=o(\sqrt{k}).
\end{align*}
Then we find
\begin{align}\label{eqs33}
o(k)=&\int_{\R^3}\epsilon^2|\nabla\varphi_k|^2+P(x)\varphi_k^2-3\mu_1\bar{S}_{\epsilon, k}^2\varphi_k^2+\int_{\R^3}\epsilon^2|\nabla\psi_k|^2+Q(x)\psi_k^2-3\mu_2\bar{T}_{\epsilon, k}^2\psi_k^2\cr
&-\beta\int_{\R^3}(\bar{S}_{\epsilon, k}^2\psi_k^2+\bar{T}_{\epsilon, k}^2\varphi_k^2+4\bar{S}_{\epsilon, k}\bar{T}_{\epsilon, k}\varphi_k\psi_k)\cr
=&k-3\mu_1\int_{\R^3}(u_{\epsilon, k}^2+U_{\epsilon,r}^2)\varphi_k^2-3\mu_2\int_{\R^3}(v_{\epsilon, k}^2+W_{\epsilon, \rho}^2)\psi_k^2-\beta\int_{\R^3}\bar{S}_{\epsilon, k}^2\psi_k^2+\bar{T}_{\epsilon ,k}^2\varphi_k^2\cr
\geq&k-(o(1)+O(e^{-\frac{R}{\epsilon}}))k-\beta\int_{\R^3}(\bar{S}_{\epsilon, k}^2\psi_k^2+\bar{T}_{\epsilon, k}^2\varphi_k^2).
\end{align}
If we choose $\beta<\frac{1}{C}$ for constant $C$ large enough, then \eqref{eqs33} is impossible for large $R$ and $k$. Consequently, we complete the proof.
\end{proof}
\begin{lem}\label{lm2.4}
There exists a constant $C>0$, independent of $k$  such that
$$
\bar{R}^{(i)}(\varphi,\psi,\xi)\leq C\|(\varphi,\psi,\xi)\|^{3-i},\ i=0,1,2.
$$
\end{lem}

\begin{proof}
These estimates can be obtained by direct computations.
\end{proof}
\begin{lem}\label{lm4.4}
There exist constants $C>0$, independent of $k$, such that
$$
\|\bar{\ell}_{k}\|\leq C\Bigl(\frac{k}{r^m}+\frac{k}{\rho^m}+(\frac{k}{r}e^{-\frac{r}{\epsilon}}+\frac{k}{\rho}e^{-\frac{\rho}{\epsilon}})+\frac{k}{r}\sum\limits_{i, j}e^{-\frac{|x^i-y^j|}{\epsilon}}\Bigr).
$$
\end{lem}
\begin{proof}
Note that
\begin{eqnarray}\label{eqs3.7}
\bar{\ell}(\varphi_{k},\psi_{k})
&=&
\int_{\R^3}\Big(\sum_{j=1}^{k}(P(x)-1)W_{\epsilon,\mu_1,x^{j}}\varphi_{k}+\sum_{j=1}^{k}(Q(x)-1)W_{\epsilon,\mu_2,y^{j}}\psi_{k}\Big)dx\cr
&&-\mu_1\int_{\R^3}\Bigl(\bar{S}_{\epsilon,k}^3-u^{3}_{\epsilon}-\sum\limits_{j=1}^kW_{\epsilon,\mu_1,x^j}^3\Bigr)\varphi_{k} -\mu_2\int_{\R^3}\Bigl(\bar{T}_{\epsilon,k}^3-v^{3}_{\epsilon}-\sum\limits_{j=1}^kW_{\epsilon,\mu_2,y^j}^3\Bigr)\psi_{k}
\cr
&&
-\beta\int_{\R^3}\Bigl(\bar{S}_{\epsilon,k}^2\bar{T}_{\epsilon,k}-u_{\epsilon}^2v_{\epsilon}\bigr)\psi_{k} dx
-\beta\int_{\R^3}\Bigl(\bar{T}_{\epsilon,k}^2\bar{S}_{\epsilon,k}-v_{\epsilon}^2u_{\epsilon}\bigr)\varphi_{k}dx\cr.
&=:&M_1+M_2+M_3.
\end{eqnarray}

We compute that
\begin{align}\label{2.2}
M_1=&\int_{\R^3}\Big(\sum_{j=1}^{k}(P(x)-1)W_{\epsilon,\mu_1,x^{j}}\varphi_{k}+(Q(x)-1)W_{\epsilon,\mu_2,y^{j}}\psi_{k}\Big)\cr
\leq& k\Big(\frac{1}{r^m}+\frac{1}{\rho^m}\Big)\|(\varphi_k,\psi_{k})\|+O(ke^{-\frac{k\pi}{\epsilon}})\|(\varphi_k,\psi_{k})\|,
\end{align}

\begin{align*}
M_2=&-\mu_1\int_{\R^3}\Bigl(\bar{S}_{\epsilon,k}^3-u^{3}_{\epsilon}-\sum\limits_{j=1}^kW_{\epsilon,\mu_1,x^j}^3\Bigr)\varphi_{k} -\mu_2\int_{\R^3}\Bigl(\bar{T}_{\epsilon,k}^3-v^{3}_{\epsilon}-\sum\limits_{j=1}^kW_{\epsilon,\mu_2,y^j}^3\Bigr)\psi_{k}\cr
\leq& C\Big(\frac{k}{r}e^{-\frac{r}{\epsilon}}+C\frac{\sqrt{k}}{r}e^{-\frac{r}{\epsilon}}\Big)\|\varphi_k\|
+\Big(\frac{k}{\rho}e^{-\frac{\rho}{\epsilon}}+C\frac{\sqrt{k}}{\rho}e^{-\frac{\rho}{\epsilon}}\Big)\|\psi_k\|,
\end{align*}
and
\begin{align}\label{2.8}
M_{3}=&\int_{\R^3}\Bigl(\bar{S}_{\epsilon,k}^2\bar{T}_{\epsilon,k}-u_{\epsilon}^2v_{\epsilon}\bigr)\psi_{k} dy+\int_{\R^3}\Bigl(\bar{T}_{\epsilon,k}^2\bar{S}_{\epsilon,k}-v_{\epsilon}^2u_{\epsilon}\bigr)\varphi_{k}dx\cr
&=C\int_{\R^3}\sum\limits_{ij}W_{\epsilon,x^i}^2W_{\epsilon,y^j}+W_{\epsilon,r}^2v_{\epsilon, k}+u_{\epsilon, k}^2W_{\epsilon,\rho}\cr
\leq&C\Big(\frac{k}{r}\sum\limits_{i, j}e^{-\frac{|x^i-y^j|}{\epsilon}}+\frac{k}{r}e^{-\frac{r}{\epsilon}}+\frac{k}{\rho}e^{-\frac{\rho}{\epsilon}}\Big)\|(\varphi_k,\psi_k)\|.
\end{align}
Combining \eqref{eqs3.7}-\eqref{2.8}, we complete the proof.
\end{proof}
\begin{prop}\label{pro4.5}
There is an integer $\varepsilon>0$ such that for each $\beta\leq\epsilon$,
there is a $C^1$ map from $D_k\times D_k $ to $\bar{E}:(\varphi,\psi)=(\varphi(r),\psi(\rho))$, $r=|x^1|,\rho=|y^1|$ satisfying $(\varphi,\psi)\in \bar{E}_k$, and
 $$
 \frac{\partial J(\varphi,\psi)}{\partial(\varphi,\psi)}\Big|_{\bar{E}_k}=0.
 $$
Moreover, there is a constant $C$, such that
$$
\|(\varphi,\psi)\|\leq C\Bigl(\frac{k}{r^m}+\frac{k}{\rho^m}+(\frac{k}{r}e^{-\frac{r}{\epsilon}}+\frac{k}{\rho}e^{-\frac{\rho}{\epsilon}})+\frac{k}{r}\sum\limits_{i, j}e^{-\frac{|x^i-y^j|}{\epsilon}}\Bigr).
$$
\end{prop}

\begin{proof}
We know that $\bar{\ell}(\varphi,\psi)$ is a bounded linear functional in $\bar{E}_k$, which implies that  there is an $\bar{\ell}_\epsilon\in \bar{E}_k$, such that
$$
\bar{\ell}(\varphi,\psi) =\langle \bar{\ell}_\epsilon,(\varphi,\psi)\rangle.
$$
Thus finding a critical point for $J(\varphi,\psi)$ is equivalent to solving
$$
\bar{\ell}_\epsilon+\bar{L}(\varphi,\psi)+\bar{R}^{'}(\varphi,\psi)=0,
$$
which is also equivalent to
$$
(\varphi,\psi)=A(\varphi,\psi)=:-\bar{L}^{-1}\bar{\ell}_\epsilon-\bar{L}^{-1}\bar{R}^{'}(\varphi,\psi),
$$
since $L$ is invertible in $\bar{E}_k$.

Set
\[
M=\Bigl\{ (\varphi,\psi)\in \bar{E}_k:\,\,\|(\varphi,\psi)\|\leq
C\Bigl(\frac{k}{r^m}+\frac{k}{\rho^m}+(\frac{k}{r}e^{-\frac{r}{\epsilon}}+\frac{k}{\rho}e^{-\frac{\rho}{\epsilon}})+\frac{k}{r}\sum\limits_{i, j}e^{-\frac{|x^i-y^j|}{\epsilon}}\Bigr)\Bigr\}.
\]

Following the standard argument, we can prove that  $A$ is a contraction map on $M$. Then by the contraction mapping theorem, we conclude that there exists $(\varphi,\psi)\in \bar{E}_k$, such that $A(\varphi,\psi)=(\varphi,\psi)$. Moreover, there holds
$$
\|(\varphi,\psi)\|\leq  C\Bigl(\frac{k}{r^m}+\frac{k}{\rho^m}+(\frac{k}{r}e^{-\frac{r}{\epsilon}}+\frac{k}{\rho}e^{-\frac{\rho}{\epsilon}})+\frac{k}{r}\sum\limits_{i, j}e^{-\frac{|x^i-y^j|}{\epsilon}}\Bigr).
$$
\end{proof}
Now we will prove Theorem \ref{th1}.

\begin{proof}[\textbf{Proof of Theorem \ref{th1}}]From Lemmas \ref{lm4.2}-\ref{lm4.4}, and Propositions \ref{lemlpoh}, we have
\begin{eqnarray*}
\bar{F}(r,\rho)
&=&
I(\bar{S}_{\epsilon,k},\bar{T}_{\epsilon,k})+\bar{\ell}(\bar{\varphi}_{k},\bar{\psi}_{k})+\frac{1}{2}\langle \bar{L}(\bar{\varphi}_{k},\bar{\psi}_{k}),(\bar{\varphi}_{k},\bar{\psi}_{k})\rangle-\bar{R}(\bar{\varphi}_{k},\bar{\psi}_{k})\nonumber\\
&=&
I(\bar{S}_{\epsilon,k},\bar{T}_{\epsilon,k})+O(\|\bar{\ell}\|\|(\bar{\varphi}_k,\bar{\psi}_k)\|+\|(\bar{\varphi}_k,\bar{\psi}_k)\|^2)\nonumber\\
&=&
I(\bar{S}_{\epsilon,k},\bar{T}_{\epsilon,k})+(\beta^2\frac{k^3}{r^2}e^{-2\sqrt{(\rho-r\cos\frac{\pi}{k})^2+(r\sin\frac{\pi}{k})^2}})
+O(\frac{1}{r^{m-1-\sigma}}+\frac{1}{\rho^{m-1-\sigma}})\nonumber\\
&=&\bar{A}+\frac{a_1k\bar{B}_1}{r^m}+\frac{a_2k\bar{B}_2}{\rho^m}-\bar{D}_1\frac{k^2}{r}e^{\frac{-r\pi}{\epsilon k}}-\bar{E}_1\frac{k^2}{\rho}e^{-\frac{\rho\pi}{\epsilon k}}-\beta\frac{k}{r}\sum\limits_{i,j}^ke^{-\frac{|x^i-y^j|}{\epsilon}}\\
&&+O(\beta^2\frac{k^3}{r^2}e^{-2\sqrt{(\rho-r\cos\frac{\pi}{k})^2+(rsin\frac{\pi}{k})^2}}+e^{-\frac{r}{\epsilon}}+e^{-\frac{\rho}{\epsilon}}).
\end{eqnarray*}
Consider the following maximization problem
\begin{equation}\label{e2.80}
\max\limits_{(r,\rho)\in D_k\times D_k}F(r,\rho).
\end{equation}
Assume that \eqref{e2.80} is achieved by some $(r^{*},\rho^{*})\in D_k\times D_k.$ We will prove that $(r^{*},\rho^{*})$ is an interior point of
$D_k\times D_k$.
We define
\begin{align*}
f(r)=\frac{a_1k\bar{B}_1}{r^m}-\bar{D}_1\frac{k^2}{r}e^{-\frac{\pi r}{\epsilon k}},\,\,\,\,\,h(\rho)=\frac{a_2k\bar{B}_2}{\rho^m}-\bar{E}_1\frac{k^2}{\rho}e^{-\frac{\rho\pi}{\epsilon k}}.
\end{align*}
Then
\begin{align*}
f'(t)=-\frac{m a_1\bar{B}_1}{t^{m+1}k^m}+\frac{\bar{D}_1}{t^2}e^{-\frac{2\pi t }{\epsilon}}+\frac{2\pi \bar{D}_1}{\epsilon}e^{-\frac{2\pi t}{\epsilon}},
\end{align*}
and
\begin{align*}
h'(s)=-\frac{m a_2\bar{B}_2}{s^{m+1}k^m}+\frac{\bar{E}_1}{s^2}e^{-\frac{2\pi s }{\epsilon}}+\frac{2\pi \bar{E}_1}{\epsilon}e^{-\frac{2\pi s}{\epsilon}}.
\end{align*}
It is easy to check that $f(t),h(s)$ have a maximum point $t_k,\,s_k$ satisfying
\begin{align*}
t_k=\big(\frac{m}{2\pi}+o(1)\big)\ln k,\,\,\,\,\,\,s_k=\big(\frac{m}{2\pi}+o(1)\big)\ln k.
\end{align*}
So the function $f(r),\,h(\rho)$ have maximum point
\begin{align*}
r_k=\big(\frac{m}{2\pi}+o(1)\big)k\ln k,\,\,\,\,\,\,\rho_k=\big(\frac{m}{2\pi}+o(1)\big)k\ln k.
\end{align*}
Hence, it follows from the expression of $\bar{F}(r,\rho)$ that the maximizer of $(r_k,\rho_k)$ is an interior point of $D_k\times D_k$.
\end{proof}

\medskip
\section{The non-degeneracy of the dichotomous synchronized peak-solutions}
In this section we are aimed to prove the non-degeneracy of the solution $(u_{\epsilon,k},v_{\epsilon,k})$
obtained in Theorem \ref{th3}.
First we revisit the
existence problem for system \eqref{eq0} and give a different proof of
Theorem \ref{th3} under the assumptions on $P(x),\,Q(x)$. Moreover we can obtain extra
point-wise estimates for the proof of non-degeneracy result.

Recall that
\begin{align*}
&x^{j}=\Big(r\cos\frac{2(j-1)\pi}{k},r\sin\frac{2(j-1)\pi}{k},0\Big),\ \ j=1,\ldots,k.
\end{align*}
Then for $k\gg1$,
\begin{align*}
&|x^{1}-x^{2}|=2r\sin\frac\pi k\sim \frac{2r\pi}k.
\end{align*}
Now we define
the weighted $L^\infty$-norm as
$$\|u\|_*=\sup_{x\in\R^3}\Big(\sum_{j=1}^me^{-\frac{\tau\theta|x-\eta^{j}|}{\epsilon}}
+\sum_{j=1}^ke^{-\frac{\tau|x-x^{j}|}{\epsilon}}\Big)^{-1}|u(x)|,$$
and
\begin{align*}
\|(u,v)\|_*=\|u\|_*+\|v\|_*,
\end{align*}
where $\tau$ is a small parameter and $\theta$ is given in Lemma \ref{lemb.1}. Here without loss of generality, we may asssume
that $\theta<1$.

Let
\begin{small}
\begin{align}\label{L}
&\mathbb L(\varphi_{1},\varphi_{2})=\left(\begin{matrix}
-\Delta\varphi_{1}+P(x)\varphi_{1}-3\mu_1S_{\epsilon,k}^2\varphi_{1}
-2\beta S_{\epsilon,k}T_{\epsilon,k}\varphi_{2}-\beta T_{\epsilon,k}^2\varphi_{1}\nonumber
\\[1mm]
-\Delta\varphi_{2}+Q(x)\varphi_{2}-3\mu_2T_{\epsilon,k}^2\varphi_{2}
-2\beta S_{\epsilon,k}T_{\epsilon,k}\varphi_{1}-\beta S_{\epsilon,k}^2\varphi_{2}
\end{matrix}\right)^\top.\nonumber
\end{align}
\end{small}

For $(f_1,f_2)\in(H_{P_1,s}\cap C(\R^3))\times (H_{P_2,s}\cap C(\R^3))$, we consider the following problem
\begin{equation}\label{eqlinear}
\mathbb L(\varphi_{1},\varphi_{2})=(f_1,f_2)+b_k\Big(\sum_{j=1}^kW_{\epsilon,x^{j}}^2Y_{j},\sum_{j=1}^kW_{\epsilon,x^{j}}^2Z_{j}\Big)
\end{equation}
with $(\varphi_{1},\varphi_{2})\in E_k$ and $b_k\in\R$, where $W_{\epsilon,x^{j}},Y_j,Z_j,E_k,j=1,\cdots,k$ are defined in Section 3.

\begin{lem}\label{lemlinear}
Assume that $(\varphi_{1,k},\varphi_{2,k})$ and $b_k$ solve
 problem \eqref{eqlinear}. Then if $\|(f_{1,k},f_{2,k})\|_*\rightarrow0$ as $k\rightarrow+\infty$, it holds
$\|(\varphi_{1,k},\varphi_{2,k})\|_*\rightarrow0$.
\end{lem}
\begin{proof}
By contradiction, we suppose that there exist $k_n\rightarrow+\infty,
r_{k_n}\in D_{k_n}$,  $(\varphi_{1,k_n},\varphi_{2,k_n})$ and
$b_{k_n}$ solving \eqref{eqlinear}, such that
$\|(f_{1,k_n},f_{2,k_n})\|_*\rightarrow0$ and
$\|(\varphi_{1,k_n},\varphi_{2,k_n})\|_*=1$. For simplicity, we drop
the subscript $k_n$.

Denote $G_i(y,x)$ the Green's function  of $-\Delta+P_i(\epsilon|y|)$ with $i=1,2$. Then by the scaling we have
\begin{align*}
\varphi_1(x)=\int_{\R^3}G_1(y,\frac{x}{\epsilon})&\Big(3\mu_1S_{\epsilon,k}^2(\epsilon y)\varphi_{1}(\epsilon y)
+2\beta  S_{\epsilon,k}(\epsilon y)T_{\epsilon,k}(\epsilon y)\varphi_{2}(\epsilon y)+
\beta T_{\epsilon,k}^2(\epsilon y)\varphi_{1}(\epsilon y)\\
&\qquad+f_1(\epsilon y)+b_k\sum_{j=1}^kW^2\big(y-\frac{x^j}{\epsilon}\big)\frac{\partial U\big(y-\frac{x^j}{\epsilon}\big)}{\partial r}\Big),
\end{align*}
and
\begin{align*}
\varphi_2(x)=\int_{\R^3}G_2(y,\frac{x}{\epsilon})&\Big(3\mu_2T_{\epsilon,k}^2(\epsilon y)\varphi_{2}(\epsilon y)
+2\beta  S_{\epsilon,k}(\epsilon y)T_{\epsilon,k}(\epsilon y)\varphi_{1}(\epsilon y)+
\beta S_{\epsilon,k}^2(\epsilon y)\varphi_{2}(\epsilon y)\\
&\qquad+f_2(\epsilon y)+b_k\sum_{j=1}^kW^2\big(y-\frac{x^j}{\epsilon}\big)\frac{\partial V\big(y-\frac{x^j}{\epsilon}\big)}{\partial r}\Big).
\end{align*}

First, we can get  that for $i=1,2$,
\begin{align*}
\Big|\int_{\R^3}G_i(y,\frac{x}{\epsilon})f_i(\epsilon y)  dy \Big|
 &=\frac{1}{\epsilon^3}\Big|\int_{\R^3}G_i(\frac{y}{\epsilon},\frac{x}{\epsilon})f_i( y)  dy \Big|\\
&\leq
C\|f_i\|_*\int_{\R^3}G_i(\frac{y}{\epsilon},\frac{x}{\epsilon})\Big(\sum_{j=1}^me^{-\frac{\tau\theta|y-\eta^{j}|}{\epsilon}}
+\sum_{j=1}^ke^{-\frac{\tau|y-x^{j}|}{\epsilon}}\Big)dy\\
&\leq C
\|f_i\|_*\Big(\sum_{j=1}^me^{-\frac{\tau\theta|x-\eta^{j}|}{\epsilon}}
+\sum_{j=1}^ke^{-\frac{\tau|x-x^{j}|}{\epsilon}}\Big).
\end{align*}
Note that
\begin{align*}
&\Big|\int_{\R^3}G_1(y,\frac{x}{\epsilon})S_{\epsilon,k}^2(\epsilon y)\varphi_{1}(\epsilon y)dy\Big|\\
&\leq C\Big|\int_{\R^3}G_1(\frac{y}{\epsilon},\frac{x}{\epsilon})
\Big(u_\epsilon+\sum_{j=1}^kU\big(\frac{y-x^j}{\epsilon}\big)\Big)^2\varphi_1( y)  dy \Big|\\
&\leq
C\|\varphi_{1}\|_*\int_{\R^3}G_1(\frac{y}{\epsilon},\frac{x}{\epsilon})
\Big(\sum_{j=1}^me^{-\frac{2\theta|y-\eta^{j}|}{\epsilon}}
+\sum_{j=1}^ke^{-\frac{2|y-x^{j}|}{\epsilon}}\Big)\Big(\sum_{j=1}^me^{-\frac{\tau\theta|y-\eta^{j}|}{\epsilon}}
+\sum_{j=1}^ke^{-\frac{\tau|y-x^{j}|}{\epsilon}}\Big)dy\\
&\leq \underbrace{C\|\varphi_{1}\|_*\int_{\R^3}G_1(\frac{y}{\epsilon},\frac{x}{\epsilon})
\Big(\sum_{j=1}^me^{-(2+\tau)\theta\frac{|y-\eta^{j}|}{\epsilon}}
+\sum_{j=1}^ke^{-(2+\tau)\frac{|y-x^{j}|}{\epsilon}}\Big)dy}_{:=K_1}\\
&\qquad+\underbrace{C\|\varphi_{1}\|_*\int_{\R^3}G_1(\frac{y}{\epsilon},\frac{x}{\epsilon})
\Big[\Big(\sum_{i=1}^me^{-\frac{2\theta|y-\eta^{i}|}{\epsilon}}\Big)\Big(\sum_{j=1}^ke^{-\frac{\tau|y-x^{j}|}{\epsilon}}\Big)
+\Big(\sum_{j=1}^ke^{-\frac{2|y-x^{j}|}{\epsilon}}\Big)\Big(\sum_{i=1}^me^{-\frac{\tau\theta|y-\eta^{i}|}{\epsilon}}\Big)\Big]dy}_{:=K_2}.
\end{align*}
It is obvious to see that
\begin{align*}
K_1\leq C
\|\varphi_1\|_*\Big(\sum_{j=1}^me^{-2\tau\theta\frac{|x-\eta^{j}|}{\epsilon}}
+\sum_{j=1}^ke^{-\frac{2\tau|x-x^{j}|}{\epsilon}}\Big),
\end{align*}
and
\begin{equation*}
   \begin{split}
K_2&\leq C\|\varphi_{1}\|_*\int_{\R^3}G_1(\frac{y}{\epsilon},\frac{x}{\epsilon})
\Big[\Big(\sum_{i,j=1}^ke^{-\frac{\tau|\eta^i-x^{j}|}{\epsilon}}\Big)\Big(\sum_{i=1}^me^{-(2\theta-\tau)\frac{|y-\eta^{i}|}{\epsilon}}\Big)\\
&\qquad\qquad\qquad\qquad\qquad\qquad\qquad+\Big(\sum_{i,j=1}^ke^{-\tau\theta\frac{|\eta^i-x^{j}|}{\epsilon}}\Big)
\Big(\sum_{j=1}^ke^{-(2-\tau\theta)\frac{|y-x^{j}|}{\epsilon}}\Big)\Big]dy\\
&\leq C\|\varphi_{1}\|_*\int_{\R^3}G_1(\frac{y}{\epsilon},\frac{x}{\epsilon})
\Big[ke^{-\frac{\tau r}{\epsilon}}\Big(\sum_{i=1}^me^{-(2\theta-\tau)\frac{|y-\eta^{i}|}{\epsilon}}\Big)
+ke^{\frac{-\tau\theta r}{\epsilon}}
\Big(\sum_{j=1}^ke^{-(2-\tau\theta)\frac{|y-x^{j}|}{\epsilon}}\Big)\Big]dy\\
&\leq C
\|\varphi_1\|_*\Big(\sum_{j=1}^me^{-2\tau\theta\frac{|x-\eta^{j}|}{\epsilon}}
+\sum_{j=1}^ke^{-\frac{2\tau|x-x^{j}|}{\epsilon}}\Big),
\end{split}
\end{equation*}
where we use the fact that $|\eta^i-x^{j}|\sim r$ for $i=1,\cdots,m,j=1,\cdots,k$ and $r\sim klnk$.
Thus from the above two estimates, we find
\begin{align}\label{4.2.1}
\Big|\int_{\R^3}G_1(y,\frac{x}{\epsilon})S_{\epsilon,k}^2(\epsilon y)\varphi_{1}(\epsilon y)dy\Big|
\leq C
\|\varphi_1\|_*\Big(\sum_{j=1}^me^{-2\tau\theta\frac{|x-\eta^{j}|}{\epsilon}}
+\sum_{j=1}^ke^{-\frac{2\tau|x-x^{j}|}{\epsilon}}\Big).
\end{align}
With the same argument as above, it holds
\begin{align*}
&\Big|\int_{\R^3}G_1(y,\frac{x}{\epsilon})(2\beta  S_{\epsilon,k}(\epsilon y)T_{\epsilon,k}(\epsilon y)\varphi_{2}(\epsilon y)+
\beta T_{\epsilon,k}^2(\epsilon y)\varphi_{1}(\epsilon y))dy\Big|\\\leq
&C\|(\varphi_{1},\varphi_{2})\|_*\Big(\sum_{j=1}^me^{-2\tau\theta\frac{|x-\eta^{j}|}{\epsilon}}
+\sum_{j=1}^ke^{-\frac{2\tau|x-x^{j}|}{\epsilon}}\Big),
\end{align*}
and
\begin{align*}
&\Big|\int_{\R^3}G_2(y,\frac{x}{\epsilon})(3\mu_2T_{\epsilon,k}^2(\epsilon y)\varphi_{2}(\epsilon y)
+2\beta  S_{\epsilon,k}(\epsilon y)T_{\epsilon,k}(\epsilon y)\varphi_{1}(\epsilon y)+
\beta S_{\epsilon,k}^2(\epsilon y)\varphi_{2}(\epsilon y))dy\Big|\\
\leq
&C\|(\varphi_{1},\varphi_{2})\|_*\Big(\sum_{j=1}^me^{-2\tau\theta\frac{|x-\eta^{j}|}{\epsilon}}
+\sum_{j=1}^ke^{-\frac{2\tau|x-x^{j}|}{\epsilon}}\Big).
\end{align*}

Moreover, direct computations imply that
\begin{align*}
\Big|\int_{\R^3}G_1(y,\frac{x}{\epsilon})\sum_{j=1}^kW^2\big(y-\frac{x^j}{\epsilon}\big)
\frac{\partial U\big(y-\frac{x^j}{\epsilon}\big)}{\partial r}\Big|\leq
C\sum_{j=1}^ke^{-\frac{|x-x^{j}|}{\epsilon}},
\end{align*}
and
\begin{align*}
\Big|\int_{\R^3}G_2(y,\frac{x}{\epsilon})\sum_{j=1}^kW^2\big(y-\frac{x^j}{\epsilon}\big)
\frac{\partial V\big(y-\frac{x^j}{\epsilon}\big)}{\partial r}\Big|\leq
C\sum_{j=1}^ke^{-\frac{|x-x^{j}|}{\epsilon}}.
\end{align*}

On the other hand, from
\begin{align*}
b_{k_n}\int_{\R^3}\Big(\sum_{j=1}^kW_{\epsilon,x^{j}}^2Y_jY_1+\sum_{j=1}^kW_{\epsilon,x^{j}}^2Z_jZ_1\Big)=
\langle\mathbb L(\varphi_1,\varphi_2)-(f_1,f_2),(Y_1,Z_1)\rangle,
\end{align*}
we can prove that $b_{k_n}\rightarrow0$. Therefore, combining the above estimates we have
\begin{align}\label{o}
|\varphi_{1,k_n}|&\Big(\sum_{j=1}^me^{-\frac{\tau\theta|x-\eta^{j}|}{\epsilon}}
+\sum_{j=1}^ke^{-\frac{\tau|x-x^{j}|}{\epsilon}}\Big)^{-1}+|\varphi_{2,k_n}|\Big(\sum_{j=1}^me^{-\frac{\tau\theta|x-\eta^{j}|}{\epsilon}}
+\sum_{j=1}^ke^{-\frac{\tau|x-x^{j}|}{\epsilon}}\Big)^{-1}\nonumber\\
&\leq o(1)+C\|(\varphi_{1,k_n},\varphi_{2,k_n})\|_*\frac{\sum_{j=1}^me^{-2\tau\theta\frac{|x-\eta^{j}|}{\epsilon}}
+\sum_{j=1}^ke^{-\frac{2\tau|x-x^{j}|}{\epsilon}}}{\sum_{j=1}^me^{-\frac{\tau\theta|x-\eta^{j}|}{\epsilon}}
+\sum_{j=1}^ke^{-\frac{\tau|x-x^{j}|}{\epsilon}}}.
\end{align}
Since $(\varphi_{1,k_n},\varphi_{2,k_n})\in  E_{k_n}$,
applying the invertibility of the linear operator $\mathbb L$ (see Lemma \ref{lm2.2}), we can prove that
  $(\varphi_{1,k_n},\varphi_{2,k_n})\rightarrow(0,0)$ uniformly in $B_{R\epsilon}(x^{j})$ for any $R>0$.
This together with \eqref{o}, gives that
 $\|(\varphi_{1,k_n},\varphi_{2,k_n})\|_*\rightarrow0$. This is a contradiction.

\end{proof}
Now we consider the following nonlinear problem. For $(\varphi_1,\varphi_2)\in E_k$ and $b_k\in\R$, we have
\begin{equation}\label{eq2.3nonlinear}
\mathbb L(\varphi_1,\varphi_2)=l_k+\tilde{R}(\varphi_1,\varphi_2)+
b_k\Big(\sum_{j=1}^kW_{\epsilon,x^{j}}^2Y_{j},\sum_{j=1}^kW_{\epsilon,x^{j}}^2Z_{j}\Big),
\end{equation}
where $l_k=(l_{1,k},l_{2,k})$ with
\begin{align*}
&l_{1,k}=\sum_{j=1}^k(P(|x|)-1)U_{\epsilon,x^{j}}-\mu_1\Big(S_{\epsilon,k}^3-u_\epsilon^3-\sum_{j=1}^kU_{\epsilon,x^{j}}^3\Big)-
\beta\Big(S_{\epsilon,k}^2T_{\epsilon,k}-u_\epsilon^2v_\epsilon-\sum_{j=1}^kU_{\epsilon,x^{j}}V_{\epsilon,x^{j}}^2\Big),\\
&l_{2,k}=\sum_{j=1}^k(Q(|x|)-1)V_{\epsilon,x^{j}}-\mu_2\Big(T_{\epsilon,k}^3-v_\epsilon^3-\sum_{j=1}^kV_{\epsilon,x^{j}}^3\Big)-
\beta\Big(T_{\epsilon,k}^2S_{\epsilon,k}-v_\epsilon^2u_\epsilon-\sum_{j=1}^kV_{\epsilon,x^{j}}U_{\epsilon,x^{j}}^2\Big),
\end{align*}
and $\tilde{R}(\varphi_1,\varphi_2)=(\tilde{R}_{1}(\varphi_1,\varphi_2),\tilde{R}_{2}(\varphi_1,\varphi_2))$ with
\begin{align*}
&\tilde{R}_{1}(\varphi_1,\varphi_2)=\mu_1(\varphi_1^3+3S_{\epsilon,k}^3\varphi_1^2)+
\beta(\varphi_1\varphi_2^2+S_{\epsilon,k}\varphi_2^2+2T_{\epsilon,k}\varphi_1\varphi_2),
\end{align*}
and
\begin{align*}
&\tilde{R}_{2}(\varphi_1,\varphi_2)=\mu_2(\varphi_2^3+3T_{\epsilon,k}^3\varphi_2^2)+
\beta(\varphi_2\varphi_1^2+T_{\epsilon,k}\varphi_1^2+2S_{\epsilon,k}\varphi_1\varphi_2).\\
\end{align*}

Following Lemma 2.2 in \cite{GZ-2022}, we also have following Lemma.
\begin{lem}\label{leml}
There hold for $i=1,2$
\begin{align*}
&\|l_{i,k}\|_*\leq\frac C{r^{m_i}}+Ce^{-\frac{|x^{1}-x^{2}|}{\epsilon}}.
\end{align*}

\end{lem}

Applying Lemma \ref{lemlinear}, we use the contraction mapping
theorem to prove that there exists an integer $k_0>0$, such that for
any integer $k\geq k_0$, \eqref{eq2.3nonlinear} has a solution
$(\varphi_1,\varphi_2)\in E_k$ satisfying
\begin{align}\label{varphinorm}
&\|(\varphi_1,\varphi_2)\|_*\leq C\|l_{k}\|_*\leq\frac C{r^{m_1}}+\frac C{r^{m_2}}+Ce^{-\frac{|x^{1}-x^{2}|}{\epsilon}}.
\end{align}

Combining Lemma \ref{lemlinear} and \eqref{varphinorm}, from Proposition \ref{propB.1}, we have

\begin{eqnarray*}
I(S_{\epsilon,k}+\varphi_1,T_{\epsilon,k}+\varphi_2)
&=&
I(S_{\epsilon,k},T_{\epsilon,k})+O(\|l_k\|_*\|(\varphi_1,\varphi_2)\|_*+\|(\varphi_1,\varphi_2)\|_*^2)\nonumber\\
&=&
k\frac{\mu_1+\mu_2-2\beta}{4(\mu_1\mu_2-\beta^2)}\ds\int_{\R^3}W_\epsilon^4dx
+k\Big(\frac{a_1\gamma_1^2}{r^{m_1}}+\frac{a_2\gamma_2^2}{r^{m_2}}\Big)\frac{1}{2}\int_{\R^3}W_\epsilon^2dx\\
&&-kC_\beta W(\frac{|x^1-x^2|}{\epsilon})+O\Big(\frac1{k^{m_1-1+\sigma}}+\frac1{k^{m_2-1+\sigma}}\Big),
\end{eqnarray*}
where $\sigma>0$ is a small constant. Therefore we can prove that problem \eqref{eq0} has a solution $(u_{\epsilon,k},v_{\epsilon,k})$ of the form
\begin{eqnarray*}
u_{\epsilon,k}=S_{\epsilon,k}(x)+\varphi_{k}(x), ~~~v_{\epsilon,k}=T_{\epsilon,k}(x)+\psi_{k}(x)
\end{eqnarray*}
with $(\varphi_k,\psi_k)\in E_k$ and as $k\rightarrow+\infty$,
\begin{align}\label{5.5}
\|(\varphi_k,\psi_k)\|_*\leq\frac C{r^{m_1}}+\frac C{r^{m_2}}+Ce^{-\frac{|x^{1}-x^{2}|}{\epsilon}}.
\end{align}

Let $(\xi_1,\xi_2)$ satisfies
\begin{equation*}\label{eq5.6}
\begin{cases}
-\epsilon^{2}\Delta \xi_1+P(|x|) \xi_{1}=3\mu_1 u_{\epsilon,k}^2\xi_{1}  +\beta \xi_{1}v_{\epsilon,k}^2 +2\beta u_{\epsilon,k} v_{\epsilon,k}\xi_2,&~\text{in}\;\mathbb R^3,\\[3mm]
-\epsilon^{2}\Delta \xi_{2}+Q(|x|) \xi_{2}= 3\mu_2 v_{\epsilon,k}^2\xi_2 +\beta u_{\epsilon,k}^2\xi_2+2\beta u_{\epsilon,k} v_{\epsilon,k}\xi_1,&~\text{in}\;\mathbb R^3.
\end{cases}
\end{equation*}
Now we want to prove Theorem \ref{th1.3} by contradiction. Suppose that there exist
$k_n\rightarrow+\infty$ and $(\xi_{1,n},\xi_{2,n})\in
(H_{\epsilon,P}\cap H^1(\R^3))\times (H_{\epsilon,Q}\cap H^1(\R^3))$ satisfying $ \|(\xi_{1,n},\xi_{2,n})\|_*=1$ and
$L_{k_n}(\xi_{1,n},\xi_{2,n})=0$.

Let
$$\tilde\xi_{i,n}(y)=\xi_{i,n}(\epsilon y+x^{1}_{k_n}).$$ We need the
following lemmas to show a contradiction.

\begin{lem}\label{lembi}
It holds that
$$(\tilde\xi_{1,n},\tilde\xi_{2,n})\rightarrow b\big(\frac{\partial U}{\partial y_1},\frac{\partial V}{\partial y_1})$$
uniformly in $C^1(B_R(0))$ for any $R>0$ with some constant $b$.
\end{lem}

\begin{proof}
By assumption we get $|\tilde\xi_{i,n}|\leq C$, and we may assume that $\tilde\xi_{i,n}\rightarrow\tilde{\xi}_i$ in $C_{loc}(\R^3)$. Then $(\tilde{\xi}_1,\tilde{\xi}_2)$
satisfies that
\begin{align}\label{eqxi2}
\begin{cases}
&\displaystyle-\Delta\tilde{\xi}_1+\tilde{\xi}_1=3\mu_1U^2\tilde{\xi}_1+\beta V^2\tilde{\xi}_1
+2\beta UV\tilde{\xi}_2,\,\,\,\,x\in\R^{3}\\[1mm]
&\displaystyle-\Delta\tilde{\xi}_2+\tilde{\xi}_2=3\mu_2V^2\tilde{\xi}_2
+\beta U^2\tilde{\xi}_2+2\beta UV\tilde{\xi}_1,\,\,\,\,\,x\in\R^{3},
\end{cases}
\end{align}
which implies that
$$(\tilde{\xi}_1,\tilde{\xi}_2)=\sum_{l=1}^3b^l\big(\frac{\partial U}{\partial y_l},\frac{\partial V}{\partial y_l}).$$
Since $\tilde{\xi}_i$ is even in $y_l,l=2,3$, it holds that $b^2=b^3=0$. This completes the proof with $b=b^1$.

\end{proof}

We decompose that
$$\xi_{i,n}(y)=b_n\sum_{j=1}^k\frac{\partial Q_{i,\epsilon,x^{j}}}{\partial r}+\xi_{i,n}^*,\ \ i=1,2,$$
where $Q_{1,\epsilon,x^{j}}=U\big(\frac{x-x^j}{\epsilon}\big), Q_{2,\epsilon,x^{j}}=V\big(\frac{x-x^j}{\epsilon}\big)$
 and $(\xi_{1,n}^*,\xi_{2,n}^*)\in E_k$. It follows from
Lemma \ref{lembi} that $b_n$ is bounded. Moreover, arguing as Lemma 3.3 in \cite{GZ-2022}, we have
\begin{align}\label{eqxi3}
\|(\xi_{1,n}^*,\xi_{2,n}^*)\|_*\leq\frac C{r^{m_1}}+\frac C{r^{m_2}}+Ce^{\frac{-|x^{1}-x^{2}|}{\epsilon}}.
\end{align}

\begin{lem}\label{lemxitilde}
There holds that
\begin{align*}
\tilde\xi_{i,n}\rightarrow0,\ i=1,2
\end{align*}
uniformly in $C^1(B_R(0))$ for any $R>0$.
\end{lem}

\begin{proof}
We apply  Lemma \ref{lemlpoh1} to $(u_{\epsilon,k},v_{\epsilon,k})$ and  $(\xi_1,\xi_2)$ with
$\Omega=B_{\frac12|x^{1}-x^{2}|}(x^{1})$. Here we drop the subscript $n$ for simipicity.
Define the following bilinear form
\begin{equation*}
 \begin{array}{rl}
\begin{split}
\mathcal L(\vec{u}_k,\vec\xi,\Omega)=&
-\epsilon^2\int_{\partial\Omega}(\frac{\partial
u_{\epsilon,k}}{\partial\nu}\frac{\partial \xi_1}{\partial y_i}+\frac{\partial
\xi_1}{\partial\nu}\frac{\partial u_{\epsilon,k}}{\partial y_i}
+\frac{\partial v_{\epsilon,k}}{\partial\nu}\frac{\partial \xi_2}{\partial y_i}+\frac{\partial \xi_2}{\partial\nu}\frac{\partial v_{\epsilon,k}}{\partial y_i})\\
&+\epsilon^2\int_{\partial\Omega}\langle\nabla
u_{\epsilon,k},\nabla\xi_1\rangle\nu_i+\langle\nabla
v_{\epsilon,k},\nabla\xi_2\rangle\nu_i
 +\int_{\partial\Omega}(u_{\epsilon,k}\xi_1+v_{\epsilon,k}\xi_2)\nu_i,
 \end{split}
 \end{array}
\end{equation*}
where $\vec u_k=(u_{\epsilon,k},v_{\epsilon,k}),\vec\xi=(\xi_1,\xi_2)$, and $\nu$ is the outward unit normal of $\partial \Omega$.
Then Lemma \ref{lemlpoh1} tells that
\begin{align}\label{0}
\begin{split}
&\mathcal L(\vec{u}_k,\vec\xi,\Omega)+\int_{\partial\Omega}((P(|x|)-1)u_{\epsilon,k}\xi_1+(Q(|x|)-1)v_{\epsilon,k}\xi_2)\nu_i\\
&-\int_{\partial\Omega}
(\mu_1 u_{\epsilon,k}^3\xi_1 +\mu_2 v_{\epsilon,k}^3\xi_2 +\beta u_{\epsilon,k} \xi_1v_{\epsilon,k}^2 +\beta \xi_2v_{\epsilon,k}u_{\epsilon,k}^2)\nu_i
=\int_\Omega\Big(\frac{\partial{P}}{\partial y_i}u_{\epsilon,k}\xi_1+\frac{\partial{Q}}{\partial y_i}v_{\epsilon,k}\xi_2\Big).
\end{split}
\end{align}
Note that we can rewrite $\mathcal L(\vec{u_k},\vec\xi,\Omega)$ as
\begin{align*}
 &\mathcal L(\vec{u}_k,\vec\xi,\Omega)\\
 &=\int_\Omega(-\epsilon^2\Delta u_{\epsilon,k}+u_{\epsilon,k}-\mu_1u_{\epsilon,k}^3-\beta u_{\epsilon,k}v_{\epsilon,k}^2)\frac{\partial \xi_1}{\partial y_1}
+\int_\Omega(-\epsilon^2\Delta v_{\epsilon,k}+v_{\epsilon,k}-\mu_2v_{\epsilon,k}^3-\beta v_{\epsilon,k}u_{\epsilon,k}^2)\frac{\partial \xi_2}{\partial y_1}\nonumber\\
&\qquad+ \int_\Omega(-\epsilon^2\Delta \xi_1+\xi_1-3\mu_1u_{\epsilon,k}^2\xi_1-2\beta u_{\epsilon,k}v_{\epsilon,k}\xi_2-\beta v_{\epsilon,k}^2\xi_1)\frac{\partial u_{\epsilon,k}}{\partial y_1}\\
&\qquad+\int_\Omega(-\epsilon^2\Delta \xi_2+\xi_2-3\mu_2v_{\epsilon,k}^2\xi_2-2\beta u_{\epsilon,k}v_{\epsilon,k}\xi_1-\beta u_{\epsilon,k}^2\xi_2)\frac{\partial v_{\epsilon,k}}{\partial y_1}\nonumber\\
&\qquad+\int_{\partial\Omega}(\mu_1u_{\epsilon,k}^3\xi_1+\mu_2v_{\epsilon,k}^3\xi_2+\beta u_{\epsilon,k}\xi_1v_{\epsilon,k}^2+\beta v_{\epsilon,k}\xi_2u_{\epsilon,k}^2)\nu_1.
\end{align*}
Recall that
\begin{align*}
u_{\epsilon,k}=u_\epsilon+\sum_{j=1}^kQ_{1,\epsilon,x^{j}}+\varphi_k,\ \ v_{\epsilon,k}=v_\epsilon+\sum_{j=1}^kQ_{2,\epsilon,x^{j}}+\psi_k,
\end{align*}
and
\begin{align*}
\xi_{i}(y)=b\sum_{j=1}^k\frac{\partial
Q_{i,\epsilon,x^{j}}}{\partial r}+\xi_{i}^*,\ \ i=1,2.
\end{align*}

Then by Lemma \ref{lemb.1} and the equations satisfied by $(Q_{1,\epsilon,x^{j}},Q_{2,\epsilon,x^{j}})$, it holds
\begin{align}\label{5.9}
\mathcal L\Big((u_\epsilon,v_\epsilon),\big(\sum_{j=1}^k\frac{\partial
Q_{1,\epsilon,x^{j}}}{\partial r},\sum_{j=1}^k\frac{\partial
Q_{2,\epsilon,x^{j}}}{\partial r}\big),\Omega\Big)
=O\big(e^{-(1-\sigma)\frac{|x^{1}-x^{2}|}{\epsilon}}\big),
\end{align}
where we use the fact that $|\eta^i-x^{j}|\sim r$ for $i=1,\cdots,m,j=1,\cdots,k$ and $r\sim klnk$.

In view of \eqref{varphinorm} and following the arguments in \cite{GZ-2022}, we have for $y\in \partial \Omega$,
$$\|(\nabla\varphi_k,\nabla\psi_k)\|_*,\|(\nabla\xi_1^*,\nabla\xi_2^*)\|_*\leq\frac C{r^{m_1}}+\frac C{r^{m_2}}+Ce^{-\frac{|x^{1}-x^{2}|}{\epsilon}}.$$
Therefore, we can check that
\begin{align}\label{2}
\begin{split}
\mathcal L((u_\epsilon,v_\epsilon),(\xi_1^*,\xi_2^*),\Omega)
=O\big(\frac C{r^{m_1}}+\frac C{r^{m_2}}+Ce^{-\frac{|x^{1}-x^{2}|}{\epsilon}}\big)e^{-(1-\sigma)\frac{|x^{1}-x^{2}|}{\epsilon}},
\end{split}
\end{align}
\begin{align}\label{3}
\begin{split}
\mathcal L((\varphi_k,\psi_k),(\xi_1^*,\xi_2^*),\Omega)
=O\big(\frac C{r^{m_1}}+\frac C{r^{m_2}}+Ce^{-\frac{|x^{1}-x^{2}|}{\epsilon}}\big)e^{-(1-\sigma)\frac{|x^{1}-x^{2}|}{\epsilon}},
\end{split}
\end{align}
\begin{align}\label{4}
\begin{split}
\mathcal L((\varphi_k,\psi_k),\big(\sum_{j=1}^k\frac{\partial
Q_{1,\epsilon,x^{j}}}{\partial r},\sum_{j=1}^k\frac{\partial
Q_{2,\epsilon,x^{j}}}{\partial r}\big),\Omega)
=O\big(\frac C{r^{m_1}}+\frac C{r^{m_2}}+Ce^{-\frac{|x^{1}-x^{2}|}{\epsilon}}\big)e^{-(1-\sigma)\frac{|x^{1}-x^{2}|}{\epsilon}},
\end{split}
\end{align}
and
\begin{align}\label{5}
\begin{split}
\mathcal L((\sum_{j=1}^kQ_{1,\epsilon,x^{j}},\sum_{j=1}^kQ_{2,\epsilon,x^{j}}),(\xi_1^*,\xi_2^*),\Omega)
=O\big(\frac C{r^{m_1}}+\frac C{r^{m_2}}+Ce^{-\frac{|x^{1}-x^{2}|}{\epsilon}}\big)e^{-(1-\sigma)\frac{|x^{1}-x^{2}|}{\epsilon}}.
\end{split}
\end{align}

On the other hand, form \cite{GZ-2022}, it holds
\begin{align}\label{1}
\begin{split}
& \mathcal L\Big(\big(\sum_{i=1}^kQ_{1,\epsilon,x^{i}},\sum_{i=1}^kQ_{2,\epsilon,x^{i}}\big),\big(\sum_{j=1}^k\frac{\partial Q_{1,\epsilon,x^{j}}}{\partial r},\sum_{j=1}^k\frac{\partial Q_{2,\epsilon,x^{j}}}{\partial r}\big),\Omega\Big)\\
&=(C+o(1))\frac{W\big(\frac{|x^1-x^2|}{\epsilon}\big)}{k^2}+
O\big(e^{-(2-\sigma)\frac{|x^{1}-x^{2}|}{\epsilon}}\big).
\end{split}
\end{align}

Finally, from the conditions imposed on $P(y),Q(y)$, we find that
\begin{align}\label{5.15}
\int_{\partial\Omega}(P(y)-1)u_{\epsilon,k}\xi_1\nu_1=O(\frac1{r^{m_i}}e^{-(1-\sigma)\frac{|x^{1}-x^{2}|}{\epsilon}}),
\end{align}
\begin{align}\label{5.15.1}
\int_{\partial\Omega}(Q(y)-1)v_{\epsilon,k}\xi_2\nu_1=O(\frac1{r^{m_i}}e^{-(1-\sigma)\frac{|x^{1}-x^{2}|}{\epsilon}}),
\end{align}
and
\begin{align}\label{5.16}
&\int_{\partial\Omega}(\mu_1u_{\epsilon,k}^3\xi_1+\mu_2v_{\epsilon,k}^3\xi_2+\beta u_{\epsilon,k}\xi_1v_{\epsilon,k}^2+\beta v_{\epsilon,k}\xi_2u_{\epsilon,k}^2)\nu_1
=O( e^{-(2-\sigma)\frac{|x^{1}-x^{2}|}{\epsilon}}).
\end{align}

Combining \eqref{0}, \eqref{5.9}-\eqref{5.16}, we get
\begin{align}\label{b1}
&\int_\Omega(\frac{\partial{P}}{\partial y_i}u_{\epsilon,k}\xi_1+\frac{\partial{Q}}{\partial y_i}v_{\epsilon,k}\xi_2)
\\&
=b(C'+o(1))\frac{W\big(\frac{|x^1-x^2|}{\epsilon}\big)}{k^2}+O\Big(\big(\frac1{r^{m_1}}+
 \frac1{r^{m_2}}\big)e^{-(1-\sigma)\frac{|x^{1}-x^{2}|}{\epsilon}}+e^{-(2-\sigma)\frac{|x^{1}-
 x^{2}|}{\epsilon}}\Big).
\nonumber\end{align}
Moreover, using \eqref{varphinorm} and Lemma \ref{lembi}, we have
\begin{equation}\label{5.19}
   \begin{split}
 &\int_\Omega(\frac{\partial{P}}{\partial y_i}u_{\epsilon,k}\xi_1+\frac{\partial{Q}}{\partial y_i}v_{\epsilon,k}\xi_2)
\\&=
b\epsilon\int_{ \mathbb R^3}\Big( Q_{1,\epsilon,x^{1}}  \frac{\partial Q_{1,\epsilon,x^{1}}}{\partial y_1} P'(|y|) |y|^{-1} y_1
 +Q_{2,\epsilon,x^{1}}  \frac{\partial Q_{2,\epsilon,x^{1}}}{\partial y^1} Q'(|y|) |y|^{-1} y_1\Big)\\&\quad
  +\frac1{|x^{1}|^{m+1}} O\Bigl(\frac{1}{ |x^{1}|^m}+
   e^{- \frac{|x^{2}-x^{1}|}{\epsilon}}\Bigr),
 \end{split}
 \end{equation}
where $m=\min\{m_1,m_2\}$. Then using \eqref{b1} and \eqref{5.19}, proceeding as the proof of Lemma 3.4 in \cite{GZ-2022}, we can prove $b=0$,
which finishes our proof.
\end{proof}

Finally, we prove Theorem \ref{th1.3}.
\begin{proof}[\textbf{Proof of Theorem \ref{th1.3}:}]

Denote $G_i(y,x)$ the Green's function  of $-\Delta+P(\epsilon|y|),\,-\Delta+Q(\epsilon|y|)$. Then by the scaling we have
\begin{align*}
\xi_{1,n}(x)=\int_{\R^3}G_1(y,\frac{x}{\epsilon})\Big(3\mu_1u_{\epsilon,k}^2(\epsilon y)\xi_{1,n}(\epsilon y)
+2\beta  u_{\epsilon,k}(\epsilon y)v_{\epsilon,k}(\epsilon y)\xi_{2,n}(\epsilon y)+
\beta v_{\epsilon,k}^2(\epsilon y)\xi_{1,n}(\epsilon y)\Big)dy,
\end{align*}
and
\begin{align*}
\xi_{2,n}(x)=\int_{\R^3}G_2(y,\frac{x}{\epsilon})\Big(3\mu_2v_{\epsilon,k}^2(\epsilon y)\xi_{2,n}(\epsilon y)
+2\beta  u_{\epsilon,k}(\epsilon y)v_{\epsilon,k}(\epsilon y)\xi_{1,n}(\epsilon y)+
\beta u_{\epsilon,k}^2(\epsilon y)\xi_{2,n}(\epsilon y)\Big)dy.
\end{align*}
Now similar to \eqref{4.2.1}, we compute that from \eqref{5.5}

\begin{equation}\label{5.20}
   \begin{split}
&\Big|\int_{\R^3}G_1(y,\frac{x}{\epsilon})u_{\epsilon,k}^2(\epsilon y)\xi_{1,n}(\epsilon y)dy\Big|\\
&\leq
C\|\xi_{1,n}\|_*\int_{\R^3}G_1(\frac{y}{\epsilon},\frac{x}{\epsilon})
\Big(S_{\epsilon,k}^2+\varphi_k^2\Big)\Big(\sum_{j=1}^me^{-\frac{\tau\theta|y-\eta^{j}|}{\epsilon}}
+\sum_{j=1}^ke^{-\frac{\tau|y-x^{j}|}{\epsilon}}\Big)dy\\
&\leq C\|\xi_{1,n}\|_*\int_{\R^3}G_1(\frac{y}{\epsilon},\frac{x}{\epsilon})
\Big[\sum_{j=1}^me^{-(2\theta-\tau)\frac{|y-\eta^{j}|}{\epsilon}}
+\sum_{j=1}^ke^{-(2-\tau\theta)\frac{|y-x^{j}|}{\epsilon}}\\
&\qquad\qquad\qquad\qquad+C\|\varphi_k\|_*\Big(\sum_{j=1}^me^{-(3\theta\tau-\sigma)\frac{|y-\eta^{j}|}{\epsilon}}
+\sum_{j=1}^ke^{-(3\tau-\sigma)\frac{|y-x^{j}|}{\epsilon}}\Big)\Big]dy\\
&\leq C\|\xi_{1,n}\|_*\Big(\sum_{j=1}^me^{-(3\theta\tau-\sigma)\frac{|x-\eta^{j}|}{\epsilon}}
+\sum_{j=1}^ke^{-(3\tau-\sigma)\frac{|x-x^{j}|}{\epsilon}}\Big),
\end{split}
 \end{equation}
where $\sigma\in (0,\theta\tau)$  is any fixed constant. Then with the same arguments as those of \eqref{5.20}, we can get that

\begin{align*}
&|\xi_{i,n}|\Big(\sum\limits_{j=1}^me^{-\frac{\tau\theta|x-\eta^{j}|}{\epsilon}}
+\sum\limits_{j=1}^ke^{-\frac{\tau|x-x^{j}|}{\epsilon}}\Big)^{-1}\nonumber\\
&\leq C\|(\xi_{1,n},\xi_{2,n})\|_*\frac{\sum\limits_{j=1}^me^{-(3\theta\tau-\sigma)\frac{|x-\eta^{j}|}{\epsilon}}
+\sum\limits_{j=1}^ke^{-(3\tau-\sigma)\frac{|x-x^{j}|}{\epsilon}}}{\sum\limits_{j=1}^me^{-\frac{\tau\theta|x-\eta^{j}|}{\epsilon}}
+\sum\limits_{j=1}^ke^{-\frac{\tau|x-x^{j}|}{\epsilon}}}.
\end{align*}
Since $\xi_{i,n}\rightarrow0$ in $B_{R\epsilon}(x^j)$ and $\|(\xi_{1,n},\xi_{2,n})\|_*=1$, it holds that
$\frac{|\xi_{i,n}|}{\sum\limits_{j=1}^me^{-\frac{\tau\theta|x-\eta^{j}|}{\epsilon}}
+\sum\limits_{j=1}^ke^{-\frac{\tau|x-x^{j}|}{\epsilon}}}$ attains
its maximum in $\mathbb R^3 \setminus \cup_{j=1}^k B_{R\epsilon}(x^j)$.
Thus
\[
\|(\xi_{1,n},\xi_{2,n})\|_* \le o(1)\|(\xi_{1,n},\xi_{2,n})\|_*,
\]
which implies a contradiction to $\|(\xi_{1,n},\xi_{2,n})\|_*=1$ and then we finish this proof.

\end{proof}

\medskip

\appendix
\section{some useful estimates}
\begin{lem}\label{lemb.1}
Assume that $(u_{\epsilon},v_{\epsilon})$ is the positive solution in Theorem A.  Then there are positive
constants $\theta$ and $C$ such that
$$
u_{\epsilon},v_{\epsilon}\leq C\sum_{j=1}^{m}e^{-\frac{\theta|x-\eta^{j}|}{\epsilon}},\,\,\,\,x\in \R^{3},
$$
and
$$
|\nabla u_{\epsilon}|,| \nabla v_{\epsilon}|\leq C e^{-\frac{\delta \theta}{4\epsilon}},\,\,\,x\in \partial B_{\delta}(\eta^{j}).
$$
\end{lem}
\begin{proof}
We know that $\|((\omega_{\epsilon},\varpi_{\epsilon}))\|=o(\epsilon^{\frac{3}{2}})$, and  $(u_\epsilon,v_\epsilon)$ satisfies
\begin{eqnarray}\begin{cases}\label{5}
-\epsilon^2\Delta u_\epsilon+(P(x)-\mu_1 u_\epsilon^2-\beta v_\epsilon^2)u_\epsilon=0,\,\,\,\,&x\in \R^{3},\cr
-\epsilon^2\Delta v_\epsilon+(Q(x)-\mu_2 v_\epsilon^2+\beta u_\epsilon^2)v_\epsilon=0,\,\,\,\,&x\in \R^{3}.
\end{cases}
\end{eqnarray}
Then following the exponential decays of $U_{\epsilon,\eta^j}$ and $V_{\epsilon,\eta^j}$,  for large fixed $R$ and $\varepsilon$ small enough, it holds
\begin{align*}
(P(x)-\mu_1 u_\epsilon^2-\beta v_\epsilon^2)\geq\frac{\lambda_1}{2},\,\,(Q(x)-\mu_2 v_\epsilon^2+\beta u_\epsilon^2)\geq \frac{\lambda_2}{2}, x\in \R^3\setminus\bigcup\limits_{j=1}^mB_{R\epsilon}(\eta^j),
\end{align*}
where $\lambda_1=\inf\limits_{x\in\R^3}P(x)$, $\lambda_2=\inf\limits_{x\in \R^3}Q(x)$. Thus
\begin{align*}
-\epsilon^2\Delta u_\epsilon+\frac{\lambda_1}{2}u_\epsilon\leq0,\,\,\,-\epsilon^2\Delta v_\epsilon+\frac{\lambda_2}{2}v_\epsilon\leq0,x\in\R^3\setminus\bigcup\limits_{j=1}^mB_{R\epsilon}(\eta^j).
\end{align*}

Now we define 
$\lambda=\max\{\lambda_1,\lambda_2\}$ and
\begin{align*}
\widetilde{L}_\epsilon w=-\epsilon^2\Delta w+\frac{\lambda}{2}w,~~~ \text{for}\,w\in H^1(\R^3).
\end{align*}
Then for $w_l=e^{-\frac{\theta|x-\eta^l|}{\epsilon}}$, where $0<\theta<\sqrt{\frac{\lambda}{2}}$ and $l\in{1,2,\cdots,m}$, we have
\begin{align*}
\tilde{L}_\epsilon w_l(x)=&-\epsilon^2(\frac{\theta^2}{\epsilon^2}-\frac{2\theta}{|x-\eta^l|})e^{-\frac{\theta|x-\eta^l|}{\epsilon}}
+\frac{\lambda}{2}e^{-\frac{\theta|x-\eta^l|}{\epsilon}}\cr
=&(\frac{2\epsilon\theta}{|x-\eta^l|}+\frac{\lambda}{2}-\theta^2)e^{-\frac{\theta|x-\eta^l|}{\epsilon}}\geq0.
\end{align*}
Next, letting $\bar{w}_l=cw_l(x)-u_\epsilon$ with $c>0$, then
\begin{align*}
\widetilde{L}_\epsilon\bar{w}_l=c\widetilde{L}_\epsilon w_l-\widetilde{L}_\epsilon u_\varepsilon\geq0,\text{in}\,\R^3\setminus\bigcup\limits_{j=1}^mB_{R\epsilon}(\eta^j).
\end{align*}
Also, since $u_\epsilon\in C(\R^3)$, there exists $c_0>0$ such that
\begin{align*}
|u_\epsilon(x)|\leq c_0\,\,\,\text{on}\,\partial(\bigcup\limits_{j=1}^mB_{R\epsilon(\eta^j)}).
\end{align*}
So taking $c=c_0e^{R\theta}$, we have $\bar{w}_l(x)=cw_l(x)-u_\epsilon\geq ce^{-R\theta}-c_0\geq0$ on $\partial(\bigcup\limits_{j=1}^mB_{R\epsilon(\eta^j)})$.
Thus for the above fixed large $R$, we obtain
\begin{eqnarray}\begin{cases}\label{b}
\widetilde{L}_\epsilon\bar{w}_l(x)\geq0,\,\,\,\hbox{in}\,\,\R^3\setminus\bigcup\limits_{j=1}^mB_{R\epsilon}(\eta^j);\cr
\bar{w}_l\geq0,\,\,\,\,\hbox{on}\,\partial(\bigcup\limits_{j=1}^mB_{R\epsilon}(\eta^j));\cr
\bar{w}_l(x)\rightarrow0,\,\,\,\,\,\hbox{as}\,|x|\rightarrow+\infty.
\end{cases}
\end{eqnarray}

Then by maximum principle, we have $\bar{w}_l\geq0$ for $x\in \R^3\setminus \bigcup\limits_{j=1}^mB_{R\epsilon}(\eta^j)$.
This gives that the results hold since we can apply the same argument to $v_\epsilon$.

\end{proof}

\begin{lem}\label{lemb.2}
Assume that $(\xi_{1\epsilon},\xi_{2\epsilon})$ solves $L_{\epsilon}(\xi_{1\epsilon},\xi_{2\epsilon})=0.$ Then it holds
$$
\|(\xi_{1\epsilon},\xi_{2\epsilon})\|=O(\epsilon^{\frac{3}{2}}).
$$
\end{lem}
\begin{proof}
Since $(\xi_{1\epsilon},\xi_{2\epsilon})$ satisfies $L_\epsilon(\xi_{1\epsilon},\xi_{2\epsilon})=0$, we have
\begin{align}\label{B.3}
-\epsilon^2(\xi_{1\epsilon}-\xi_{2\epsilon})+(P(x)\xi_{1\epsilon}-Q(x)\xi_{2\epsilon})=&u_\epsilon^2(3\mu_1\xi_{1\epsilon}
-\beta\xi_{2\epsilon})+v_\epsilon^2(\beta \xi_{1\epsilon}-3\mu_2\xi_{2\epsilon})\cr
&+2\beta u_\epsilon v_\epsilon(\xi_{1\epsilon}-\xi_{2\epsilon}).
\end{align}
Denoting $\eta_\epsilon=\xi_{1\epsilon}-\xi_{2\epsilon}$, then $\eta_\epsilon$  satisfies
\begin{align}\label{B.3.1}
\|(\xi_{1\epsilon},\xi_{2\xi})\|\leq C\int_{\R^3}(u_\epsilon^2+v_\epsilon^2)\eta_\epsilon^2.
\end{align}
Noting that $(u_\epsilon,v_\epsilon)=\Big(\sum\limits_{j=1}^{m}U(\frac{y-\eta^{j}}{\epsilon})+\omega_{\epsilon}(y),
 \sum\limits_{j=1}^{m}V(\frac{y-\eta^{j}}{\epsilon})+\varpi_{\epsilon}(y)\Big)$ and
\begin{align}\label{A.4}
\int_{\R^3}U^2_{\epsilon,\eta^j}\eta_\epsilon^2dx\leq C\epsilon^3,\,\,\int_{\R^3}V^2_{\epsilon,\eta^j}\eta_\epsilon^2dx\leq C\epsilon^3,
\end{align}
\begin{align}\label{A.5}
\int_{\R^3}\omega_\epsilon^2\eta_\epsilon^2dx\leq \Big(\int_{\R^3}\varphi_\epsilon^6dx\Big)^{\frac{1}{3}}\Big(\int_{\R^3}\eta_\epsilon^3dx\Big)^{\frac{2}{3}}\leq C\epsilon\|\eta_\epsilon\|_\epsilon^{\frac{4}{3}},
\end{align}
and
\begin{align}\label{A.5}
\int_{\R^3}\varpi_\epsilon^2\eta_\epsilon^2dx\leq C\epsilon\|\eta_\epsilon\|_\epsilon^{\frac{4}{3}},
\end{align}
where $\|\eta_\epsilon\|_\epsilon=\Big(\ds\int_{\R^3}\epsilon^2|\eta_\epsilon|^2+\eta^2\Big)^\frac{1}{2,}$
from \eqref{B.3.1} to \eqref{A.5}, we have
\begin{align*}
\|\eta_\epsilon\|_\epsilon^2\leq C(\epsilon^3+\epsilon\|\eta_\epsilon\|_\epsilon^{\frac{4}{3}}),
\end{align*}
which completes the proof.
\end{proof}

\begin{lem}\label{lemb.3}
Assume that $(\xi_{1\epsilon},\xi_{2\epsilon})$ solves $L_{\epsilon}(\xi_{1\epsilon},\xi_{2\epsilon})=0.$
Then there hold that
constants $\theta$ and $C$ such that
$$
|\xi_{i \epsilon}(x)| \leq C\sum_{j=1}^{m}e^{-\frac{\theta|x-\eta^{j}|}{\epsilon}},\,\,\,\,x\in \R^{3},\,\,\,i=1,2,j=1,\cdots,m,
$$
and
$$
|\xi_{i \epsilon}(x)|\leq C e^{-\frac{\delta \theta}{4\epsilon}},\,\,\,x\in \partial B_{\delta}(\eta^{j})\,\,\,\,i=1,2,j=1,\cdots,m
$$
\end{lem}
\begin{proof}
As we know $(\xi_{1\epsilon},\xi_{2\epsilon})$ satisfies the equation $L_\epsilon(\xi_{1\xi},\xi_{2\xi})=0$, that is
\begin{eqnarray}\begin{cases}\label{B.2}
-\varepsilon^2\Delta\xi_{1\varepsilon}+P(x)\xi_{1\varepsilon}=3\mu_1u_\epsilon^2\xi_{1\epsilon}+\beta\xi_{1\epsilon}v_\epsilon^2+2\beta u_\epsilon v_\epsilon\xi_{2\epsilon},\cr
-\varepsilon^2\Delta\xi_{2\varepsilon}+Q(x)\xi_{2\varepsilon}=3\mu_2v_\epsilon^2\xi_{2\epsilon}+\beta\xi_{2\epsilon}u_\epsilon^2+2\beta u_\epsilon v_\epsilon\xi_{1\epsilon}.
\end{cases}
\end{eqnarray}
By Lemma A.1, we know there is large fixed $R\gg1$ such that
\begin{align}\label{A.7}
(P(x)-3\mu_1u_\epsilon^2-\beta u_\epsilon v_\epsilon)\geq\frac{\lambda}{2},\,\,(Q(x)-3\mu_2v_\epsilon^2-\beta u_\epsilon v_\epsilon)\geq\frac{\lambda}{2},\,\,\hbox{in}\,\R^3\setminus B_{R\epsilon}(\eta^{j}).
\end{align}
where $\lambda=\inf\limits_{x\in\R^3}\{P(x),Q(x)\}$. It is assumed that
\begin{align*}
 D_1:=\{x\in\R^3\setminus B_{R\epsilon}(\eta^{j}):(\xi_{1\epsilon}-\xi_{2\xi})\geq0\}.
\end{align*}
From \eqref{B.2}-\eqref{A.7}, we obtain
\begin{align*}
-\epsilon^2\Delta(\xi_{1\epsilon}-\xi_{2\epsilon})+\frac{\lambda}{2}(\xi_{1\epsilon}-\xi_{2\epsilon})\leq(3\mu_1\xi_{1\epsilon}-\beta\xi_{2\epsilon})u_\epsilon^2
+(\beta\xi_{1\epsilon}-3\mu_2\xi_{2\epsilon})v_\epsilon^2=:f_\epsilon(x).
\end{align*}
Denote $\eta_\epsilon$ be a solution of the equation $-\epsilon^2\Delta\eta_\epsilon+\frac{\lambda}{2}\eta_\epsilon=f_\epsilon$ and $g_\epsilon(x)=(\xi_{1\epsilon}-\xi_{2\epsilon})-\eta_\epsilon$.
Then $g_\epsilon(x)$ satisfies
\begin{align*}
-\epsilon^2\Delta g_\epsilon+\frac{\lambda}{2}g_\epsilon\leq0.
\end{align*}
Setting $\bar{\eta}_\epsilon(x)=\eta_\epsilon(\epsilon x)$, then $\bar{\eta}_\epsilon(x)$ solves
\begin{align*}
-\Delta\bar{\eta}_\epsilon+\frac{\lambda}{2}\bar{\eta}_\epsilon=f_\epsilon(\epsilon x).
\end{align*}
Thus, we have
\begin{align*}
\bar{\eta}_\epsilon=\int_{\R^3}G^{\sqrt{\frac{\lambda}{2}}}(x-z)f_{\epsilon}(\epsilon z)dz,\hbox{where}\,\,G^{\sqrt{\frac{\lambda}{2}}}(x)=\frac{1}{4\pi|x|}e^{-\sqrt{\frac{\lambda}{2}}|x|}.
\end{align*}
Also, there holds
\begin{align*}
\eta_\epsilon(x)=\bar{\eta}_\epsilon(\frac{x}{\epsilon})=\int_{\R^3}G^{\sqrt{\frac{\lambda}{2}}}(\frac{x}{\epsilon}-z)f(\epsilon z)dz.
\end{align*}
By Holder inequality, we have for any $x\in D_1$,
\begin{align*}
|\xi_{i\epsilon}|\leq C|\eta_\epsilon|=&\int_{\R^3}\frac{1}{4\pi|\frac{x}{\epsilon}-z|}e^{-\sqrt{\frac{\lambda}{2}}|\frac{x}{\epsilon}-z|}f(\epsilon z)dz\cr
=&\int_{\R^3}\frac{1}{4\pi|\frac{x}{\epsilon}-z|}e^{-\sqrt{\frac{\lambda}{2}}|\frac{x}{\epsilon}-z|}((3\mu_1\xi_{1\epsilon}-\beta\xi_{2\epsilon})u_\epsilon^2
+(\beta\xi_{1\epsilon}-3\mu_2\xi_{2\epsilon})v_\epsilon^2)\cr
\leq& Ce^{-\frac{\theta|x-\eta^{j}|}{\epsilon}}\|(\xi_{1\epsilon},\xi_{2\epsilon})\|\|(u_\epsilon,v_\epsilon)\|^2\leq Ce^{-\frac{\theta|x-\eta^{j}|}{\epsilon}}.
\end{align*}
Define an operator
\begin{align*}
\tilde{L}_\epsilon v:=-\epsilon^2\Delta v+\frac{\lambda}{2}v.
\end{align*}
Letting $v_1(x)=e^{-\frac{\theta|x-\eta^{j}|}{\epsilon}}$, then we have
\begin{align*}
\tilde{L}_\epsilon v_1(x)=\Big(\frac{\lambda}{2}-\theta^2+\frac{2\epsilon \theta}{|x-\eta^{j}|}\Big)e^{-\frac{\theta|x-\eta^{j}|}{\epsilon}}>0.
\end{align*}
Setting $g_\epsilon'(x)=Me^{\theta R}e^{-\frac{\theta|x-\eta^{j}|}{\epsilon}}-g_\epsilon(x)$ with a large constant $M>0$, it holds
\begin{align*}
\tilde{L}_\epsilon g_\epsilon'\geq0,\,\,\hbox{in}\, D_1.
\end{align*}
If $x\in \partial B_{R \epsilon}(\eta^{j})$, we easily get $g_\epsilon'(x)\geq M-g_\epsilon\geq0$.
If $(\xi_{1\epsilon},\xi_{2\epsilon})=(0,0)$, then we have
\begin{align*}
g_\epsilon'(x)\geq Me^{\theta R}e^{-\frac{\theta|x-\eta^{j}|}{\epsilon}}-\eta_\epsilon(x)\geq0.
\end{align*}
Hence for $x\in \partial D_1$, it has $g_\epsilon'(x)\geq0$. By the maximum principle, we have $g_\epsilon'(x)\geq0,\forall x\in D_1$. Then for any $x\in D_1$, we know
\begin{align*}
|\xi_{i\epsilon}|\leq |\eta_\epsilon(x)|+Me^{\theta R}e^{-\frac{\theta|x-\eta^{j}|}{\epsilon}}-|\eta_\epsilon|\leq Ce^{-\frac{\theta|x-\eta^{j}|}{\epsilon}}.
\end{align*}

Now we consider the case
\begin{align*}
x\in D_2=\{x\in\R^3\setminus B_{R \epsilon}(\eta^{j}):~~~ \xi_{1\epsilon}+\xi_{2\epsilon}<0\}.
\end{align*}
From \eqref{B.2}, we can deduce that
\begin{align*}
-\epsilon^2\Delta(\xi_{1\epsilon}-\xi_{2\epsilon})+\frac{\lambda}{2}(\xi_{1\epsilon}-\xi_{2\epsilon})\geq f_\epsilon(x).
\end{align*}
Similarly, we have $|\xi_{1\epsilon}|+|\xi_{2\epsilon}|\geq Ce^{-\frac{\theta|x-\eta^{j}|}{\epsilon}}$. Moreover, for any $x\in B_{R\epsilon}(\eta^{j})$
\begin{align*}
|\xi_{i\epsilon}|\leq C\leq Ce^{\theta R}e^{-\frac{\theta|x-\eta^{j}|}{\epsilon}}.
\end{align*}
On the other hand, by \eqref{A.7}, we obtain
\begin{align*}
-\epsilon^2\Delta(\xi_{1\epsilon}-\xi_{2\epsilon})+\frac{\lambda}{2}(\xi_{1\epsilon}-\xi_{2\epsilon})\leq(3\mu_1\xi_{1\epsilon}-\beta\xi_{2\epsilon})u_\epsilon^2
+(\beta\xi_{1\epsilon}-3\mu_2\xi_{2\epsilon})v_\epsilon^2.
\end{align*}
Using $L^p-$ estimate , for any $z\in \partial B_\delta(\eta^{j})$, we have
\begin{align*}
\|(\xi_{1\epsilon},\xi_{2\epsilon})\|\leq Ce^{-\frac{\delta\theta}{4\epsilon}}.
\end{align*}
\end{proof}

\section{ Estimates of the energy functional}
In this section, we mainly estimate $I(S_{\epsilon, k }, T_{\epsilon, k})$ and $I(\bar{S}_{\epsilon, k }, \bar{T}_{\epsilon, k})$.
\begin{prop}\label{propB.1}
For $k$ large enough, we have
\begin{align*}
I(S_{\epsilon, k},T_{\epsilon, k})=&A+\big(\frac{a_1\gamma_1^2k}{2r^{m_1}}+\frac{a_2\gamma_2^2k}{2r^{m_2}}\big)\int_{\R^3}W_\epsilon^2
-C_\beta\frac{k}{r}\sum\limits_{j=2}^ke^{-\frac{|x^j-x^1|}{\epsilon}}-\frac{k^2}{r}e^{\frac{-r}{\epsilon}}
+O(ke^{-\frac{\pi r}{\epsilon k}})\cr
=&kA_0+k\big(\frac{a_1\gamma_1^2}{r^{m_1}}+\frac{a_2\gamma_2^2}{r^{m_2}}\big)B-\frac{k^2}{r}e^{-\frac{r\pi}{\epsilon k}}+O(ke^{-\frac{\pi r}{\epsilon k}}),
\end{align*}
where $A_0=\frac{\mu_1+\mu_2-2\beta}{4(\mu_1\mu_2-\beta^2)}\ds\int_{\R^3}W_\epsilon^4$,\,$B=\ds\frac{1}{2}\int_{\R^3}W_\epsilon^2$,\,$C_\beta>0$ is a constant depending on $\beta$, and
$W_\epsilon=W(\frac{x}{\epsilon})$.
\end{prop}
\begin{proof}
From \eqref{eqs3.1}, we have
\begin{align}\label{1}
I(S_{\epsilon, k},T_{\epsilon, k})=&A+\frac{1}{2}\int_{\R^3}\Big((P(x)-1)U_{\epsilon, r}^2+(Q(x)-1)V_{\epsilon, r}^2\Big)\cr
&-\frac{\mu_1}{4}\int_{\R^3}\Big(S_{\epsilon, k}^4-u_{\epsilon}^4-U_{\epsilon, r}^4-4u_\epsilon^3U_{\epsilon, r}-2\sum\limits_{i\neq j}U_{\epsilon, x^i}^3U_{\epsilon, x^j}\Big)\cr
&-\frac{\mu_2}{4}\int_{\R^3}\Big(T_{\epsilon, k}^4-v_{\epsilon}^4-V_{\epsilon, r}^4-4v_\epsilon^3V_{\epsilon, r}-2\sum\limits_{i\neq j}V_{\epsilon, x^i}^3V_{\epsilon, x^j}\Big)\cr
&-\frac{\beta}{2}\int_{\R^3}\Big(S_{\epsilon, k}^2T_{\epsilon, k}^2-u_\epsilon^2v_\epsilon^2-u_\epsilon v_\epsilon U_{\epsilon, r}-u_\epsilon^2v_\epsilon V_{\epsilon, r}-U_{\epsilon,r}^2V_{\epsilon,r}^2\Big)\cr
&-\frac{\mu_1}{4}\int_{\R^3}\Big(U_{\epsilon, r}^4-\sum\limits_{j=1}^kU_{\epsilon, x^j}^4-2\sum\limits_{i\neq j}U_{\epsilon, x^i}^3U_{\epsilon, x^j}\Big)\cr
&-\frac{\mu_2}{4}\int_{\R^3}\Big(V_{\epsilon, r}^4-\sum\limits_{j=1}^kV_{\epsilon, x^j}^4-2\sum\limits_{i\neq j}V_{\epsilon, x^i}^3V_{\epsilon, x^j}\Big)\cr
&-\frac{\beta}{2}\int_{\R^3}\Big(U_{\epsilon, r}^2V_{\epsilon, r}^2-\sum\limits_{i=1}^kU_{\epsilon,x^i}^2V_{\epsilon,x^i}^2-\sum\limits_{i\neq j}^kV_{\epsilon x^i}U_{\epsilon,x^i}U_{\epsilon,x^j}-\sum\limits_{i\neq j}U_{\epsilon,x^i}^2V_{\epsilon, x^i}V_{\epsilon, x^j}\Big)\cr
=:&A+I_1-I_2-I_3-I_4-I_5-I_6-I_7,
\end{align}
where $A=k\Big(\frac{1}{2}\ds\int_{\R^3}|\nabla U_\epsilon|^2+U_\epsilon^2+|\nabla V_\epsilon|^2+V_\epsilon^2-\frac{1}{4}\int_{\R^3}\mu_1U_\epsilon^4+\mu_2V_\epsilon^4-\frac{\beta}{2}\int_{\R^3}U_\epsilon^2V_\epsilon^2\Big)$,
$U_{\epsilon,r}=\sum_{j=1}^kU_{\epsilon,x^j}$ and $V_{\epsilon,r}=\sum_{j=1}^kV_{\epsilon,x^j}$.

Since $(U_\epsilon,V_\epsilon)=(\gamma_1W_\epsilon,\gamma_2W_\epsilon)$, we have
\begin{align}\label{C1}
A=\frac{\mu_1+\mu_2-2\beta}{4(\mu_1\mu_2-\beta^2)}k\int_{\R^3}W_\epsilon^4.
\end{align}
By the exponential decay of $U$, it holds
\begin{align}\label{C.1}
&\int_{\R^3}(P(x)-1)U_{\epsilon, r}^2dx=\int_{\R^3}(P(x)-1)\big(\sum\limits_{j=1}^kU_{\epsilon,x^j}^2+2\sum\limits_{i\neq j}U_{\epsilon,x^i}U_{\epsilon,x^j}\big)\cr
=&\frac{a_1k}{r^{m_1}}\int_{\R^3}U^2\big(\frac{x}{\epsilon}\big)dx+O\Big(\sum\limits_{j=2}^ke^{-\frac{|x^1-x^j|}{2\epsilon}}\Big)
=\frac{a_1\gamma_1^2k}{r^{m_1}}\int_{\R^3}W_\epsilon^2dx+O\big(\sum\limits_{j=2}^ke^{-\frac{|x^1-x^j|}{2\epsilon}}\big).
\end{align}
Similarly, we can estimate
\begin{align*}
\int_{\R^3}(Q(x)-1)V_{\epsilon, r}^2dx=\frac{a_2k\gamma_2^2}{r^{m_2}}\int_{\R^3}W_\epsilon^2dx+O\big(\sum\limits_{j=2}^ke^{-\frac{|x^1-x^j|}{2\epsilon}}\big).
\end{align*}
Following the above two estimates, there holds
\begin{align}\label{C.2}
I_1=\int_{\R^3}(P(x)-1)U_{\epsilon, r}^2+(Q(x)-1)V_{\epsilon, r}^2dx=(\frac{a_1\gamma_1^2k}{2r^{m_1}}+\frac{a_2\gamma_2^2k}{2r^{m_2}})\int_{\R^3}W_\epsilon^2dx
+O\big(\sum\limits_{j=2}^ke^{-\frac{|x^1-x^j|}{2\epsilon}}\big).
\end{align}
On the other hand, we have
\begin{align}\label{C.3}
I_2=&\int_{\R^3}\Big(S_{\epsilon, k}^4-u_{\epsilon}^4-U_{\epsilon, r}^4-4u_\epsilon^3U_{\epsilon, r}-2\sum\limits_{i\neq j}U_{\epsilon, x^i}^3U_{\epsilon, x^j}\Big)\cr
=&O\Big(\int_{\R^3}u_{\epsilon}^2U^2_{\epsilon, r}+u_\epsilon^3U_{\epsilon, r}+\sum\limits_{i\neq j}U_{\epsilon, x^i}^2U_{\epsilon, x^j}^2\Big)\cr
=&O\big(\frac{k^2}{r}e^{-\frac{r}{\epsilon}}+\frac{k}{r}\sum\limits_{i\neq j}e^{-\frac{|x^i-x^j|}{\epsilon}}\big).
\end{align}
The same arguments as \eqref{C.3} give that
\begin{align}\label{C.4}
I_3=\int_{\R^3}\Big(T_{\epsilon, k}^4-v_{\epsilon}^4-V_{\epsilon, r}^4-4v_\epsilon^3V_{\epsilon, r}-2\sum\limits_{i\neq j}V_{\epsilon, x^i}^3V_{\epsilon, x^j}\Big)
=O\big(\frac{k^2}{r}e^{-\frac{r}{\epsilon}}+\frac{k}{r}\sum\limits_{i\neq j}e^{-\frac{|x^i-x^j|}{\epsilon}}\big),
\end{align}
\begin{align}\label{C.5}
I_5=\int_{\R^3}\Big(U_{\epsilon, r}^4-\sum\limits_{j=1}^kU_{\epsilon, x^j}^4-2\sum\limits_{i\neq j}U_{\epsilon, x^i}^3U_{\epsilon, x^j}\Big)=O\Big(\int_{\R^3}\sum\limits_{i\neq j}U_{\epsilon, x^i}^3U_{\epsilon, x^j}\Big)=O\big(\frac{k}{r}\sum\limits_{i\neq j}e^{-\frac{|x^i-x^j|}{\epsilon}}\big),
\end{align}
and
\begin{align}\label{C.6}
I_6=O\big(\frac{k}{r}\sum\limits_{i\neq j}e^{-\frac{|x^i-x^j|}{\epsilon}}\big),\,\,\,\,\,\,\,I_7=O\big(C_\beta\frac{k}{r}\sum\limits_{i\neq j}e^{-\frac{|x^i-x^j|}{\epsilon}}\big).
\end{align}
\end{proof}
Combining \eqref{C.1}-\eqref{C.6}, we obtain the result.

\begin{prop}\label{lemlpoh}
There holds
\begin{align*}
I(\bar{S}_{\epsilon, k }, \bar{T}_{\epsilon, k})=&\bar{A}+\big(\frac{a_1k\gamma_1^2}{2r^m}+\frac{a_2k\gamma_2^2}{2\rho^m}\big)\int_{\R^3}W_\epsilon^2dx-\bar{D}_1\frac{k^2}{r}e^{\frac{-r\pi}{\epsilon k}}-\bar{D}_2\frac{k^2}{r}e^{-\frac{r}{\epsilon}}\cr
&-\bar{E}_1\frac{k^2}{\rho}e^{-\frac{\rho\pi}{\epsilon k}}-\bar{E}_2\frac{k^2}{\rho}e^{-\frac{\rho}{\epsilon}}-\beta\frac{k}{r}\sum\limits_{i,j}^ke^{-\frac{|x^i-y^j|}{\epsilon}}+O(ke^{-\frac{\pi \rho}{\epsilon k}})\cr
=&\bar{A}+\frac{a_1k\bar{B}_1}{r^m}+\frac{a_2k\bar{B}_2}{\rho^m}-\bar{D}_1\frac{k^2}{r}e^{\frac{-r\pi}{\epsilon k}}-\bar{D}_2\frac{k^2}{r}e^{-\frac{r}{\epsilon}}-\bar{E}_1\frac{k^2}{\rho}e^{-\frac{\rho\pi}{\epsilon k}}\cr
&-\bar{E}_2\frac{k^2}{\rho}e^{-\frac{\rho}{\epsilon}}-\beta\frac{k}{r}\sum\limits_{i,j}^ke^{-\frac{|x^i-y^j|}{\epsilon}},
\end{align*}
where $\bar{A}=\ds\frac{1}{4}\int_{\R^3}\Big(\mu_1\sum\limits_{j=1}^kW_{\epsilon,\mu_1, x^j}^4+\mu_2\sum\limits_{j=1}^kW_{\epsilon,\mu_2, y^j}^4\Big)$, $\bar{B}_i=\frac{\gamma_i^2}{2}\ds\int_{\R^3}W_{\epsilon}^2dx $, $(i=1,2)$, $\bar{D}_1,\,\bar{D}_2,\,\bar{E}_1,\,\bar{E}_2$ are positive constants.
\end{prop}
\begin{proof}
Since  $W_{\epsilon,\mu_i,y^j}$ satisfies \eqref{eqs1.2},
one has
\begin{align}\label{11}
I(\bar{S}_{\epsilon k},\bar{T}_{\epsilon k})=&\frac{1}{4}\int_{\R^3}\Big(\mu_1\sum\limits_{j=1}^kW_{\epsilon,\mu_1, x^j}^4+\mu_2\sum\limits_{j=1}^kW_{\epsilon,\mu_2, y^j}^4\Big)
+\frac{1}{2}\int_{\R^3}\Big((P(x)-1)W_{\epsilon, r}^2+(Q(x)-1)W_{\varepsilon, \rho}^2\Big)\cr
&-\frac{1}{4}\mu_1\int_{\R^3}\Big((u_{\epsilon}+W_{\epsilon, r})^4-u_{\epsilon}^4-\sum\limits_{j=1}^kW_{\epsilon,\mu_1,x^j}^4-4u_{\epsilon}^3W_{\epsilon,r}-2\sum\limits_{i\neq j}W_{\epsilon,\mu_1,x^i}^3W_{\epsilon,\mu_1,x^j}\Big)\cr
&-\frac{1}{4}\mu_1\int_{\R^3}\Big((v_{\epsilon}+W_{\epsilon, \rho})^4-v_{\epsilon}^4-\sum\limits_{j=1}^kW_{\varepsilon,\mu_2,y^j}^4-4v_{\epsilon}^3W_{\epsilon,\rho}-2\sum\limits_{i\neq j}W_{\epsilon,\mu_2,y^i}^3W_{\epsilon,\mu_2,y^j}\Big)\cr
&-\frac{\beta}{2}\int_{\R^3}\Big((u_{\epsilon}+W_{\epsilon, r})^2(v_{\epsilon}+W_{\epsilon,\rho})^2-u_{\epsilon }^2v_{\epsilon }^2-2u_{\epsilon}v_{\epsilon }^2W_{\epsilon,r}-2u_{\epsilon }^2v_{\epsilon }W_{\epsilon,\rho}\Big)\cr
&=:\bar{A}+I_1-\frac{1}{4}I_2-\frac{1}{4}I_3-\frac{\beta}{2}I_4,
\end{align}
where $W_{\epsilon,r}=\sum_{j=1}^kW_{\epsilon,\mu_1, x^j}$ and $W_{\epsilon,\rho}=\sum_{j=1}^kW_{\epsilon,\mu_2, y^j}$.

Similar as Proposition B.1, we have
\begin{align}\label{C.2}
I_1=&\int_{\R^3}\Big((P(x)-1)W_{\epsilon, r}^2+(Q(x)-1)W_{\epsilon, \rho}^2\Big)dx\cr
&=\big(\frac{a_1k\gamma_1^2}{r^m}+\frac{a_2k\gamma_2^2}{\rho^m}\big)\int_{\R^3}W_\epsilon^2dx
+O\Big(\sum\limits_{j=2}^ke^{-\frac{|x^1-x^j|}{2\epsilon}}\Big),
\end{align}
\begin{align}\label{C.8}
I_2=C\frac{k}{r}\sum\limits_{j=2}^ke^{-\frac{|x^1-x^j|}{\epsilon}}+C\frac{k^2}{r}e^{-\frac{r}{\epsilon}}+O(ke^{-\frac{\pi r}{\epsilon k}}),\,I_3=C\frac{k}{\rho}\sum\limits_{j=2}^ke^{-\frac{|y^1-y^j|}{\epsilon}}+C\frac{k^2}{\rho}e^{\frac{-\rho}{\epsilon}}+O(ke^{-\frac{\pi \rho}{\epsilon k}}),
\end{align}
and
\begin{align}\label{8}
I_4=&\int_{\R^3}\Big((u_{\epsilon}+W_{\epsilon,r})^2(v_{\epsilon}+W_{\epsilon,\rho})^2
-u_{\epsilon}^2v_{\epsilon}^2-2u_{\epsilon}v_{\epsilon}^2W_{\epsilon,r}-2u_{\epsilon}^2v_{\epsilon}W_{\epsilon,\rho}\Big)\cr
=&O\Big(\frac{k}{r}\sum\limits_{i,j}e^{-\frac{|x^i-y^j|}{\epsilon}}+ke^{-\frac{\pi r}{\epsilon k}}\Big).
\end{align}

Using \eqref{11}-\eqref{8}, the result follows.
\end{proof}

\medskip

\noindent{\bf Acknowledgements.}

\noindent  This work is supported by NSFC (No.12126345 and No.12471106). The authors would like to thank Professor Chunhua Wang from Central
China Normal University for the helpful discussion with her.

\noindent{\bf  Declarations}

\noindent Conflict of interest:\, All authors declare that they have no Conflict of interest.

\noindent Ethical approval:\,  This article does not contain any studies with human participants or animals performed by
the authors.

\noindent Data availability:\,  There is no data in our article.


\end{document}